\documentclass[11pt,amssymb]{amsart}
\usepackage{amssymb,hyperref}
\usepackage[mathscr]{eucal}
\usepackage[all,cmtip]{xy}
\usepackage{amscd}
\usepackage{enumitem}
\usepackage{tikz-cd}

\newcommand{\im}{{\rm im}\:}

\newtheorem{thm}{Theorem}[section]
\newtheorem{lemma}[thm]{Lemma}
\newtheorem{prop}[thm]{Proposition}
\newtheorem{cor}[thm]{Corollary}

\theoremstyle{definition}

\newtheorem{example}[thm]{Example}
\newtheorem{remark}[thm]{Remark}

.240pk scaled 1200 .240pk
\usepackage[all,cmtip]{xy}
\usepackage[T2A,T1]{fontenc}
\DeclareSymbolFont{cyrillic}{T2A}{cmr}{m}{n}
\DeclareMathSymbol{\Sha}{\mathalpha}{cyrillic}{216}

\begin{document}

\title{Products of trees and  ${\rm PGL}_2$-torsors over the punctured affine line}

\author[A.S.~Rapinchuk]{Andrei S. Rapinchuk}

\address{Department of Mathematics, University of Virginia, Charlottesville, VA 22904-4137, USA}

\email{asr3x@virginia.edu}

\author[I.A.~Rapinchuk]{Igor A. Rapinchuk}

\address{Department of Mathematics, Michigan State University, East Lansing, MI 48824, USA}

\email{rapinchu@msu.edu}

\author[A.~Roy]{Avinash Roy}

\address{Department of Mathematics, Michigan State University, East Lansing, MI 48824, USA}

\email{royavina@msu.edu}

\begin{abstract}

The goal of this article is to present a computation of the Galois cohomology of the group $G = \mathrm{PGL}_2$ over rings of Laurent polynomials. This computation recovers the main result of \cite{CGP} in this case by a new method, based on the analysis of actions on appropriate geometric objects (products of trees and, in general,  of affine buildings), which was already used in \cite{ARR} to give a new proof of the theorem of Raghunathan-Ramanathan \cite{RR} concerning Galois cohomology over polynomial rings. Going beyond Galois cohomology, this method also enables one to determine the finite subgroups in the group of points over relevant polynomial rings. In view of these and other potential applications of the method in different situations (in particular, in the study of algebraic groups over the coordinate rings of general affine curves), we have attempted to make our exposition largely self-contained and accessible to broad mathematical audience.
\end{abstract}

\vskip5mm

\maketitle

\section{Introduction and statement of results}

The theory of Bruhat-Tits buildings, initially introduced in
\cite{BT72}, \cite{BT84}, and \cite{BT87}, and recently re-developed in \cite{KaPr} from a different perspective, provides powerful techniques for establishing various general results about
the Galois cohomology of reductive algebraic groups over {\it local} fields. Broadly speaking, these techniques rely on the analysis of actions of Galois groups, constructed using Galois cocycles, on appropriate buildings, together with applications of the Bruhat-Tits Fixed Point Theorem.
Recently, it was shown in \cite{ARR} that similar considerations can be used in the {\it global}  situation to study the Galois cohomology of reductive algebraic groups over polynomial rings in one variable, and, in particular, to offer a short conceptual proof of a theorem concerning torsors over the affine line due to Raghunathan-Ramanathan \cite{RR}. In addition, the same approach was used to give
a description of finite subgroups in the groups of points of reductive groups over polynomial rings. One can expect that these considerations
will be applicable in a significantly more general situation, including in the study of reductive groups over the coordinate rings of various open subsets of the projective line as well as of curves of higher genus.

The purpose of this article is to give a practically self-contained exposition of the method in order to facilitate further progress in this direction. More specifically, we will demonstrate how these techniques can be used to recover for $G = \mathrm{PGL}_2$ the result of \cite{CGP} concerning torsors under reductive groups over the affine line with one puncture. In order to keep our exposition accessible, we will not use the terminology of algebraic geometry and will instead formulate our results in terms of Galois cohomology over rings of Laurent polynomials --- the interested reader is referred to \cite[\S 5.1]{ARR} for a discussion of the connections between these results on Galois cohomology
and torsors.

The method we present here relies on constructing a nice fundamental domain for an action of the group of points over a relevant ring (which, in the current paper is the ring of Laurent polynomials) on a relevant geometric object (which is a product of two Bruhat-Tits trees in the present case), and then determining the stabilizers of points in this fundamental domain
--- the underlying formalism is described in Lemma \ref{L:cohom}.
It should be pointed out that, quite generally,
the construction of a fundamental domain is the subject matter of {\it reduction theory} (see \cite[Chapter 4]{PlRR} for an account of reduction theory for arithmetic groups). Major contributions to reduction theory over global function fields were made by G.~Harder (cf.
\cite{H-Curve}, \cite{H-Mink}); for some
current efforts in this direction, the reader may want to consult \cite{ACBLLA}, \cite{BL1}, and \cite{BL2}, and the references therein. In the present paper, we
avoid the technical results of reduction theory and instead construct a suitable fundamental domain explicitly using the realization of Bruhat-Tits trees in terms of lattices in 2-dimensional vector spaces, which goes back to Serre \cite[\S 2.1]{Serre-Trees} --- see \S \ref{S:BT}.

Let us now
turn to the precise statements of the results we prove. Let $G = \mathrm{PGL}_2$ over a field $k$ of characteristic $\neq 2$. For a Galois extension $\ell/k$ with Galois group $\mathcal{G} = \mathrm{Gal}(\ell/k)$, we set $\Gamma_{\ell}$ to be the group $G(\ell[x , x^{-1}])$ of points of $G$ over the ring $\ell[x , x^{-1}]$ of Laurent polynomials with the standard $\mathcal{G}$-action.  Our goal is to describe the Galois cohomology set $H^1(\mathcal{G} , \Gamma_{\ell})$, which we will often denote by
$H^1(\ell/k , \Gamma_{\ell})$. For $\ell/k$ finite, we consider the map
$$
\nu_{\ell} \colon H^1(\ell/k , \Gamma_{\ell}) \longrightarrow H^1(\ell((x))/k((x)) , G),
$$
where $k((x))$ (resp., $\ell((x))$) is the field of formal Laurent series over $k$ (resp., over $\ell$), induced by the  natural embedding $\Gamma_{\ell} \hookrightarrow G(\ell((x)))$. Taking the direct limit of these maps over all Galois extensions $\ell$ of $k$ contained in a fixed algebraic closure $k^{\mathrm{sep}}$, we obtain a map
$$
\nu_{k^{\mathrm{sep}}} \colon H^1(k^{\mathrm{sep}}/k , \Gamma_{k^{\mathrm{sep}}}) \longrightarrow H^1(k((x))^{\mathrm{ur}}/k((x)) , G),
$$
where $k((x))^{\mathrm{ur}} := \bigcup_{\ell} \ell((x))$ is the maximal unramified extension of $k((x))$.

\begin{thm}\label{T:Main-sep}
The map $\nu_{k^{\mathrm{sep}}}$ is a bijection.
\end{thm}


We derive Theorem \ref{T:Main-sep} from the following.

\begin{thm}\label{T:Main-quadratic}
For every quadratic extension $\ell/k$, the map $\nu_{\ell}$ is a bijection.
\end{thm}

An important point is that our proofs
of Theorems \ref{T:Main-sep} and \ref{T:Main-quadratic} do not make use of
any ``global'' results on Galois cohomology over the field of rational functions $K = k(x)$, such as Faddeev's sequence.

\vskip1mm

The paper is structured as follows.
In \S \ref{S:BT}, we review the construction of the Bruhat-Tits tree associated with a local field $\mathcal{K}$, calculate a fundamental domain and the stabilizers of points in this fundamental domain under the action of $G(\mathcal{K}) = \mathrm{PGL}_2(\mathcal{K})$, and use this information to compute and interpret the Galois cohomology of $G$ over $\mathcal{K}$.
In \S \ref{S:Product}, we consider the field of rational functions $K = k(x)$, its valuations $v$ and $v^-$ associated with $x$ and $x^{-1}$, respectively, the corresponding completions $k((x))$ and $k((x^{-1}))$, and the associated Bruhat-Tits trees $\mathcal{T}$ and $\mathcal{T}^-$. The goal of this section is to describe a fundamental domain and the stabilizers of points in this fundamental domain for the diagonal action of $\Gamma = G(k[x , x^{-1}])$ on the product $\mathcal{X} = \mathcal{T} \times \mathcal{T}^-$. In \S \ref{S:Proof1.2} we apply these results to the action of $\Gamma_{\ell}$ for a given quadratic extension $\ell/k$ on $\mathcal{X}_{\ell} = \mathcal{T}_{\ell} \times \mathcal{T}_{\ell}^-$ to prove Theorem \ref{T:Main-quadratic} (where $\mathcal{T}_{\ell}$ and $\mathcal{T}_{\ell}^-$ are the Bruhat-Tits trees associated with $\ell((x))$ and $\ell((x^{-1}))$). Then in \S \ref{S:proof1.1}, we prove Theorem \ref{T:Main-sep}. The last two sections are devoted to some
other applications of the results developed in \S \ref{S:Product}. Namely, in \S  \ref{S:FS}, we describe finite subgroups of $\Gamma = G(k[x , x^{-1}])$ for  a field $k$ of characteristic zero (Theorem \ref{T:FS1}), which implies that when $k$ is a $p$-adic field, $\Gamma$ has finitely many conjugacy classes of finite subgroups (Corollary \ref{C:FC1}). Finally, in \S \ref{S:LG}, the Raghunathan-Ramanathan Theorem and our Theorem \ref{T:Main-sep} are applied to prove the triviality of the kernel of the global-to-local map in Galois cohomology in some situations (see Theorems \ref{T:LG1} and \ref{T:LG2}).

Throughout the paper, we use standard notations pertaining to algebraic number theory, algebraic groups, and Galois cohomology  (see \cite{PlRR}, particularly the list of notations on pages xiii-xv).

\section{The Bruhat-Tits tree}\label{S:BT}

In this section, we first review the construction of the Bruhat-Tits tree associated with the group $G = \mathrm{PGL}_2$ over a local field $\mathcal{K}$ (following \cite{Serre-Trees}). We then describe in detail the Bruhat-Tits approach for computing
the Galois cohomology of $G$ over $\mathcal{K}$. Although these results are known, to the best of our knowledge, they have not been previously discussed in the literature in the context of Bruhat-Tits theory. We will use these ideas in later sections to prove Theorems   \ref{T:Main-sep} and \ref{T:Main-quadratic}.

\subsection{Construction} Let $G = \mathrm{PGL}_2$, and suppose $\mathcal{K}$ is a field equipped with a discrete valuation $v$. Then the group $G(\mathcal{K}) = \mathrm{PGL}_2(\mathcal{K})$ naturally acts on a tree $\mathcal{T}$ (which is precisely the Bruhat-Tits building associated with this situation). We begin by quickly reviewing the explicit construction of this tree, as described in \cite[Ch. II, \S 1]{Serre-Trees}.

Let $\mathcal{O}$ be the valuation ring of $\mathcal{K}$ with valuation ideal $\mathfrak{p}$, and fix a uniformizer $\pi \in \mathfrak{p}.$
We consider a 2-dimensional $\mathcal{K}$-vector space $\mathcal{W} = \mathcal{K}^2$ with a (fixed) standard basis $e_1 , e_2$. A {\it lattice} (or $\mathcal{O}$-lattice) $\Lambda$ in $\mathcal{W}$ is an $\mathcal{O}$-submodule $\Lambda \subset \mathcal{W}$ of rank two. For a lattice $\Lambda$, we denote by $[\Lambda]$ the (equivalence) class of lattices that are proportional to $\Lambda$. Then
$$
\mathbf{V} := \{ \ [\Lambda] \ : \ \Lambda \subset \mathcal{W} \ \ \text{is a lattice} \}
$$
will serve as the vertex set of our tree $\mathcal{T}$. A pair of vertices $\mathcal{P}_1 , \mathcal{P}_2 \in \mathbf{V}$ defines a (nonoriented) edge in $\mathcal{T}$ if there exist lattices $\Lambda_1 , \Lambda_2 \subset \mathcal{W}$ such that $\mathcal{P}_i = [\Lambda_i]$ for $i = 1, 2,$ and $\Lambda_1 \subset \Lambda_2$, with $\Lambda_1/\Lambda_2 \simeq \mathcal{O}/\mathfrak{p}$ as $\mathcal{O}$-modules. Let $\mathbf{E}$ denote the set of all edges.
\begin{lemma}
{\rm (\cite[Ch. II, \S 1]{Serre-Trees})} $\mathcal{T} = (\mathbf{V} , \mathbf{E})$ is a tree.
\end{lemma}
\noindent {\it Quick proof.} \underline{Connectedness}: Given any two vertices $\mathcal{P} , \mathcal{Q}$, we can find lattices $\Lambda , \Lambda' \subset \mathcal{W}$ representing these vertices so that $\Lambda \supset \Lambda'$. There exists $m \geq 1$ such that $\pi^m \Lambda \subset \Lambda'$, so the quotient $\Lambda/\Lambda'$ is an $\mathcal{O}$-module of finite length. Hence, one can find a finite sequence of lattices
$$
\Lambda = \Lambda_0 \supset \Lambda_1 \supset \cdots \supset \Lambda_n = \Lambda'
$$
with $\Lambda_i/\Lambda_{i+1} \simeq \mathcal{O}/\mathfrak{p}$ for all $i = 0, \ldots , n-1$. Then the vertices $\mathcal{P}_i = [\Lambda_i]$ define a path $\mathcal{P} = \mathcal{P}_0, \mathcal{P}_1, \ldots , \mathcal{P}_n = \mathcal{Q}$ in $\mathcal{T}$ connecting $\mathcal{P}$ and $\mathcal{Q}$.

\vskip1mm

\underline{Absence of cycles}: Suppose the vertices $\mathcal{P}_0, \ldots , \mathcal{P}_n = \mathcal{P}_0$ $(n \geq 2)$ define a cycle without backtracking in $\mathcal{T}$. We can choose lattices $\Lambda_i \subset \mathcal{W}$ $(i = 0, \ldots , n)$ such that $\mathcal{P}_i = [\Lambda_i]$, and $\Lambda_i \supset \Lambda_{i+1}$ with $\Lambda_i/\Lambda_{i+1} \simeq \mathcal{O}/\mathfrak{p}$. Since $\mathcal{P}_n = \mathcal{P}_0$, we have $\Lambda_n = \lambda \Lambda_0$ for some $\lambda \in \mathfrak{p}$, and therefore the quotient $\Lambda_0/\Lambda_n$ is not a cyclic $\mathcal{O}$-module. Let $\ell \geq 1$ be maximal such that $\Lambda_0/\Lambda_{\ell}$ is cyclic --- clearly $\ell < n$. Then $\Lambda_0/\Lambda_{\ell+1}$ is not cyclic, so it follows from the theorem on elementary divisors that $\Lambda' := \pi^{-1}\mathcal{L}_{\ell+1}$ is contained in $\Lambda_0$, and it obviously contains $\Lambda_{\ell}$ (since $\pi \Lambda_{\ell} \subset \Lambda_{\ell+1}$) with $\Lambda'/\Lambda_{\ell} \simeq \mathcal{O}/\mathfrak{p}$. The cyclicity of $\Lambda_0/\Lambda_{\ell}$ implies that $\Lambda' = \Lambda_{\ell-1}$.
Thus, the vertices $\mathcal{P}_{\ell-1} , \mathcal{P}_{\ell} , \mathcal{P}_{\ell+1}$ form  a backtracking, a contradiction.~\hfill ~ $\Box$

\vskip2mm

In this note, we will routinely identify the combinatorial graph $\mathcal{T}$ with its geometric realization.

\vskip2mm

For $H = \mathrm{GL}_2$, the standard action of $H(\mathcal{K})$ on $\mathcal{W}$ induces a transitive action on the set of all lattices, and also on $\mathbf{V}$. Moreover, the action of $H(\mathcal{K})$ on $\mathbf{V}$ descends to an action of $G(\mathcal{K})$. It is easy to see that $G(\mathcal{K})$ acts transitively on the edges of $\mathcal{T}$; in other words, every edge is a (weak) fundamental domain for this action. More specifically, let us consider the following lattices:
$$
\Lambda_n = \mathcal{O}e_1 + \mathcal{O} \pi^n e_2, \ \ \ n = 0, 1, \ldots ,
$$
for which the corresponding vertices $\mathcal{P}_i = [\Lambda_i]$ form a ray in $\mathcal{T}$; then the segment $\mathcal{P}_0\mathcal{P}_1$ is a fundamental domain for the action of $G(\mathcal{K})$ on $\mathcal{T}$.

\subsection{Stabilizers} The goal of this subsection is to record a description of the stabilizers of certain vertices and some other points in $\mathcal{T}$ under the action of $G(\mathcal{K})$. For $A \in \mathrm{GL}_2(\mathcal{K})$, we denote by $[A]$ the corresponding element of $G(\mathcal{K}) = \mathrm{PGL}_2(\mathcal{K})$. We also set
$$
\mathrm{PGL}_2(\mathcal{O}) = \{ \ [A] \ \vert \ A \in \mathrm{GL}_2(\mathcal{O}) \ \} \ \ \text{and} \ \ \Sigma_n = \mathrm{Stab}_{G(\mathcal{K})}([\Lambda_n]).
$$
If $A \in \mathrm{GL}_2(\mathcal{K})$ and $[A] \in \Sigma_0$, then $A(\Lambda_0) = \lambda \Lambda_0$ for some $\lambda \in \mathcal{K}^{\times}$. So, $B := \lambda^{-1}A \in \mathrm{GL}_2(\mathcal{O})$ and $[A] = [B]$, implying that $\Sigma_0 = \mathrm{PGL}_2(\mathcal{O})$. Furthermore, for any $n \geq 0$, the matrix $T_n := \mathrm{diag}( 1 , \pi^n)$ satisfies $T_n(\Lambda_0) = \Lambda_n$, so
\begin{equation}\label{E:StabDesc1}
\Sigma_n = [T_n] \cdot \mathrm{PGL}_2(\mathcal{O}) \cdot [T_n]^{-1} \ \ \text{for all} \ \ n \geq 0.\end{equation}
Set
$$
W_n = \left(  \begin{array}{cc} 0 & 1 \\ \pi^{2n+1} & 0 \end{array} \right).
$$
One easily checks that
$$
W_n(\Lambda_n) = \pi^n \Lambda_{n+1}, \ \ \ W_n(\Lambda_{n+1}) = \pi^{n+1} \Lambda_n, \ \ \ \text{and} \ \ \ W_n^2 = \pi^{2n+1} I_2.
$$
So, $w_n = [W_n]$ is an element of order two in $G(\mathcal{K})$ that interchanges $[\Lambda_n]$ and $[\Lambda_{n+1}]$.

Let $\mathcal{M}_n$ be the midpoint of the edge $\mathcal{E}_n$ in (the geometric realization of) $\mathcal{T}$ connecting $[\Lambda_n]$ and $[\Lambda_{n+1}]$. Then for any point $\mathcal{P} \in \mathcal{E}_n$ other than $[\Lambda_n]$, $[\Lambda_{n+1}]$, and $\mathcal{M}_n$, we have
$$
\Sigma(\mathcal{P}) = \mathrm{Stab}_{G(\mathcal{K})}(\mathcal{P}) = \Sigma_n \bigcap \Sigma_{n+1},
$$
while
$$
\Sigma(\mathcal{M}_n) = \mathrm{Stab}_{G(\mathcal{K})}(\mathcal{M}_n) = (\Sigma_n \bigcap \Sigma_{n+1}) \rtimes \langle w_n \rangle.
$$
We will now record the matrix descriptions of these stabilizers.
\begin{prop}\label{P:Stab1}
In the above notations, we have

\vskip2mm

\begin{itemize}
\item[{\rm (i)}] $\Sigma_n = \mathrm{Stab}_{G(\mathcal{K})}([\Lambda_n]) = \left\{ [A] \ \colon \ A = \left( \begin{array}{cc} a & b \\ c & d  \end{array} \right), \ \ a , d \in \mathcal{O}, \ c \in  \pi^n \mathcal{O}, \ b \in \pi^{-n}\mathcal{O}, \ \text{and}  \ \det A \in \mathcal{O}^{\times}   \right\}$;

\vskip2mm

\item[{\rm (ii)}] $\Sigma_n \cap \Sigma_{n+1} = \left\{ [A] \ \colon \ A = \left( \begin{array}{cc} a & b \\ c & d \end{array}  \right), \  \  a , d \in \mathcal{O}, \ c \in \pi^{n+1} \mathcal{O}, \ b \in \pi^{-n} \mathcal{O}, \ \text{and} \ \det A \in \mathcal{O}^{\times}    \right\}$;

\vskip2mm

\item[{\rm (iii)}] $\Sigma(\mathcal{M}_n) = \mathrm{Stab}_{G(\mathcal{K})}(\mathcal{M}_n) = (\Sigma_n \bigcap \Sigma_{n+1}) \rtimes \langle w_n \rangle$, and $\Sigma(\mathcal{P}) = \mathrm{Stab}_{G(\mathcal{K})}(\mathcal{P}) = \Sigma_n \bigcap  \Sigma_{n+1}$ for any point $\mathcal{P} \in \mathcal{E}_n$ different from $[\Lambda_n]$, $[\Lambda_{n+1}]$, and $\mathcal{M}_n$.

\end{itemize}





\end{prop}
\begin{proof}
(i): One easily checks that
\begin{equation}\label{E:StabDesc2}
 T_n \cdot \mathrm{GL}_2(\mathcal{O}) \cdot T_n^{-1} =  \left\{ [A] \ \colon \ A = \left( \begin{array}{cc} a & b \\ c & d  \end{array} \right), \ \ a , d \in \mathcal{O}, \ c \in  \pi^n \mathcal{O}, \ b \in \pi^{-n}\mathcal{O}, \ \text{and}  \ \det A \in \mathcal{O}^{\times}   \right\},
\end{equation}
and our claim follows from (\ref{E:StabDesc1}).

\vskip1mm

(ii): Let $g \in \Sigma_n \bigcap \Sigma_{n+1}$. Then
$$
g = [A] = [B] \ \ \text{with} \ \ A \in T_n \cdot \mathrm{GL}_2(\mathcal{O}) \cdot T_{n}^{-1} \ \ \text{and} \ \ B \in T_{n+1} \cdot \mathrm{GL}_2(\mathcal{O}) \cdot T_{n+1}^{-1}.
$$
In this case, $A = \lambda B$ with $\lambda \in \mathcal{O}^{\times}$, so it follows from (\ref{E:StabDesc2}) for $n$ and $n+1$ that
$$
g \in \left\{ [A] \ \colon \ A = \left( \begin{array}{cc} a & b \\ c & d \end{array}  \right), \  \  a , d \in \mathcal{O}, \ c \in \pi^{n+1} \mathcal{O}, \ b \in \pi^{-n} \mathcal{O}, \ \text{and} \ \det A \in \mathcal{O}^{\times}    \right\},
$$
proving one inclusion, with the other inclusion being obvious.

\vskip1mm

(iii): This has already been established.
\end{proof}

\vskip1mm

\subsection{Cohomology} For a field $F$, we denote by $Q(F)$ the set of isomorphism classes of central quaternion $F$-algebras. It is well known (cf. \cite[Theorem 2.4.3]{GS}) that for $G = \mathrm{PGL}_2$, there is a natural bijection $\kappa_F \colon H^1(F , G) \to Q(F)$ that sends the cohomology class of a Galois cocycle $\zeta \in Z^1(F , G)$ to the isomorphism class of the twisted algebra $\ _{\zeta}\mathrm{M}_2$ (where we identify $G$ with the automorphism group of the matrix algebra $\mathrm{M}_2$). Furthermore, for a fixed (separable) quadratic extension $E/F$,
we set $Q(E/F)$ to be the set of isomorphism classes of quaternion $F$-algebras that contain $E$ as a maximal subfield (or, equivalently, split over $E$). Then the above map $\kappa_F$ induces a bijection $\kappa_{E/F} \colon H^1(E/F , G) \to Q(E/F)$.

Now let $\mathcal{K}$ be a local field (i.e. a field that is complete with respect to a discrete valuation) with the residue field $k$, and assume that ${\rm char}~k \neq 2.$
It is well known that
due to this assumption,
every quaternion $\mathcal{K}$-algebra $\mathcal{D}$ contains a maximal subfield $\mathcal{L}$ that is unramified over $\mathcal{K}$ (this follows, for example, from \cite[Theorem 10.1]{Saltman}); let $\ell$ denote the residue field of $\mathcal{L}$. Then we have a decomposition
\begin{equation}\label{E:Decomp1}
Q(\mathcal{L}/\mathcal{K}) = Q(\mathcal{L}/\mathcal{K})_{\mathrm{ur}} \, \cup \, Q(\mathcal{L}/\mathcal{K})_{\mathrm{ram}},
\end{equation}
where $Q(\mathcal{L}/\mathcal{K})_{\mathrm{ur}}$ (resp., $Q(\mathcal{L}/\mathcal{K})_{\mathrm{ram}}$) is the subset of $Q(\mathcal{L}/\mathcal{K})$ consisting of the isomorphism classes of unramified (resp., ramified) quaternion algebras.
Our goal in this subsection is to identify the decomposition of $H^1(\mathcal{L}/\mathcal{K} , G)$ that corresponds to (\ref{E:Decomp1}) under the bijection $\kappa_{\mathcal{L}/\mathcal{K}} \colon H^1(\mathcal{L}/\mathcal{K}, G) \to Q(\mathcal{L}/\mathcal{K})$, and in fact to produce this decomposition using an appropriate Bruhat-Tits tree.

For this, we extend the 2-dimensional $\mathcal{K}$-vector space $\mathcal{W}$ with the standard basis $e_1 , e_2$ to the 2-dimensional $\mathcal{L}$-vector space $\mathcal{W}_{\mathcal{L}} = \mathcal{W} \otimes_{\mathcal{K}} \mathcal{L}$ and equip $\mathcal{W}_{\mathcal{L}}$ with the natural action of the Galois group $\mathcal{G} = \mathrm{Gal}(\mathcal{L}/\mathcal{K}) = \{1, \sigma \}$. Then the Bruhat-Tits tree $\mathcal{T}_{\mathcal{L}}$ constructed using 
$\mathcal{O}_{\mathcal{L}}$-lattices in $\mathcal{W}_{\mathcal{L}}$ (where $\mathcal{O}_{\mathcal{L}}$ is the valuation ring in $\mathcal{L}$) also receives a $\mathcal{G}$-action. We note that sending the class $[\Lambda]$ of an $\mathcal{O}$-lattice $\Lambda$ to the class $[\Lambda \otimes_{\mathcal{O}} \mathcal{O}_{\mathcal{L}}]$ defines an isometric embedding $\mathcal{T} \to \mathcal{T}_{\mathcal{L}}$, and that the image of this embedding (which will be identified with $\mathcal{T})$) coincides with the fixed point set $\bigl( \mathcal{T}_{\mathcal{L}}  \bigr)^{\mathcal{G}}$. (Indeed, if the nontrivial element $\sigma \in \mathcal{G}$ fixes the equivalence class $[\tilde{\Lambda}]$ of an $\mathcal{O}_{\mathcal{L}}$-lattice $\tilde{\Lambda}$, then it actually leaves invariant the lattice $\tilde{\Lambda}$ itself, and therefore $\tilde{\Lambda} = \Lambda \otimes_{\mathcal{O}} \mathcal{O}_{\mathcal{L}}$ for the $\mathcal{O}$-lattice $\Lambda := \bigl( \tilde{\Lambda}  \bigr)^{\mathcal{G}}$ by unramified descent.)

The group $G(\mathcal{L})$ also acts on $\mathcal{T}_{\mathcal{L}}$, and this action is compatible with the natural action of $\mathcal{G}$ on both objects, leading to an action of the semi-direct product $\Gamma = G(\mathcal{L}) \rtimes \mathcal{G}$. It follows from the above discussion that the segment (edge) $\mathcal{P}_0\mathcal{P}_1$, where $\mathcal{P}_n = [\mathcal{O}_{\mathcal{L}} e_1 +  \mathcal{O}_{\mathcal{L}} \pi^n e_2]$ for $n = 0, 1$, is a fundamental domain for the action of $G(\mathcal{L})$ on $\mathcal{T}_{\mathcal{L}}$.

Now let $\zeta \in Z^1(\mathcal{L}/\mathcal{K} , G)$ be a Galois 1-cocycle on $\mathcal{G} = \mathrm{Gal}(\mathcal{L}/\mathcal{K})$ with values in $G(\mathcal{L})$. Then the map
$$
\mathcal{G} \to \Gamma, \ \ \ \tau \mapsto (\zeta(\tau) , \tau)
$$
is a group homomorphism (see \cite[Lemma 3.1]{ARR}), which gives rise to an isometric action of $\mathcal{G}$ on $\mathcal{T}_{\mathcal{L}}$, given by
$$
\mathcal{P} \mapsto \zeta(\tau)(\tau(\mathcal{P})).
$$
By the Bruhat-Tits Fixed Point Theorem (see \cite[Theorem 11.23]{AbBr} or \cite[\S3.2]{BT72}),
this action has a fixed point $\mathcal{Q} \in \mathcal{T}_{\mathcal{L}}$. Since the segment (edge) $\mathcal{P}_0 \mathcal{P}_1$ is a fundamental domain, there exist $\mathcal{P} \in \mathcal{P}_0\mathcal{P}_1 \setminus \{ \mathcal{P}_1 \}$ and $g \in G(\mathcal{L})$ such that $\mathcal{Q} = g \mathcal{P}$. Then the equivalent cocycle $\zeta'$ defined by
$$
\zeta'(\tau) = g^{-1} \zeta(\tau) \tau(g) \ \ \text{for} \ \ \tau \in \mathcal{G}
$$
takes values in the stabilizer $\Sigma_{\mathcal{L}}(\mathcal{P}) := \mathrm{Stab}_{G(\mathcal{L})}(\mathcal{P})$ (see Lemma \ref{L:cohom} below for a more general statement). The above computations of stabilizers shows that for any $\mathcal{P} \in \mathcal{P}_0\mathcal{P}_1 \setminus \{ \mathcal{M} , \mathcal{P}_1 \}$, we have the inclusion
$$
\Sigma_{\mathcal{L}}(\mathcal{P}) \subseteq \Sigma_{\mathcal{L}}(\mathcal{P}_0) = \mathrm{PGL}_2(\mathcal{O}_{\mathcal{L}}).
$$
Thus, we obtain the decomposition
$$
H^1(\mathcal{L}/\mathcal{K} , G) \, = \, \mathcal{H}_0 \, \cup \, \mathcal{H}_1,
$$
where $\mathcal{H}_0$ (resp., $\mathcal{H}_1$) is the image of the map $H^1(\mathcal{G} , \Sigma_{\mathcal{L}}(\mathcal{P}_0)) \to H^1(\mathcal{L}/\mathcal{K} , G)$ (resp., of the map $H^1(\mathcal{G} , \Sigma_{\mathcal{L}}(\mathcal{M})) \to H^1(\mathcal{L}/\mathcal{K} , G)$) --- we note that $\mathcal{P}_0 , \mathcal{M} \in \mathcal{T}$, hence are fixed by $\mathcal{G}$, making their stabilizers $\mathcal{G}$-invariant.

The set $\mathcal{H}_0$ is easy to identify. Let $\zeta \in Z^1(\mathcal{L}/\mathcal{K} , \mathrm{PGL}_2(\mathcal{O}_{\mathcal{L}}))$ be a 1-cocycle, and
set
$\mathcal{D} = {}_{\zeta} \mathrm{M}_2(\mathcal{K})$ and $\mathcal{A} = {}_{\zeta} \mathrm{M}_2(\mathcal{O})$, noting that $\mathrm{PGL}_2(\mathcal{O}_{\mathcal{L}})$ acts on $\mathrm{M}_2(\mathcal{O}_{\mathcal{L}})$ by automorphisms. Then $\mathcal{A} \otimes_{\mathcal{O}} \mathcal{O}_{\mathcal{L}} \simeq \mathrm{M}_2(\mathcal{O}_{\mathcal{L}})$, implying that $\mathcal{A}$ is an Azumaya $\mathcal{O}$-algebra, and $\mathcal{D} = \mathcal{A} \otimes_{\mathcal{O}} \mathcal{K}$. The latter means that $\mathcal{D}$ is unramified; in fact, every unramified quaternion $\mathcal{K}$-algebra that splits over $\mathcal{L}$ is obtained this way, i.e. it is a twist of $\mathrm{M}_2$ by a Galois cocycle on $\mathcal{G}$ with values in $\mathrm{PGL}_2(\mathcal{O}_{\mathcal{L}})$ (see, e.g., \cite{IR} for further details).
We note that by Hensel's Lemma (see \cite[\S 3.4, Lemme 2(ii)]{BT87}), the elements of $H^1(\mathcal{L}/\mathcal{K} , \mathrm{PGL}_2(\mathcal{O}_{\mathcal{L}}))$ are in bijection with those of $H^1(\ell/k , \mathrm{PGL}_2(\ell))$, which parametrize quaternion algebras over $k$ that split over $\ell$. On the other hand, the map $$H^1(\mathcal{L}/\mathcal{K} , \mathrm{PGL}_2(\mathcal{O}_{\mathcal{L}})) \to H^1(\mathcal{L}/\mathcal{K} , \mathrm{PGL}_2(\mathcal{L}))$$ is injective, which follows, for example, from a general result of Nisnevich (\cite[Th\'eor\`eme 4.2]{Nisn}). Thus, the map $\kappa_{\mathcal{L}/\mathcal{K}}$ yields a bijection between $\mathcal{H}_0$ and $Q(\mathcal{L}/\mathcal{K})_{\mathrm{ur}}$. (We note that this observation includes the almost tautological fact that there is a natural bijection between the isomorphisms classes of {\it unramified} quaternion $\mathcal{K}$-algebras that split over $\mathcal{L}$ and the isomorphism classes of quaternion $k$-algebras that split over $\ell$).

\vskip2mm

We now turn to the set
$\mathcal{H}_1$, which consists of cohomology classes represented by Galois cocycles on $\mathcal{G}$ with values in $\Sigma_{\mathcal{L}}(\mathcal{M}) = \Sigma \rtimes \langle w \rangle$, where
$$
\Sigma = \left\{ [A] \ \colon A = \left( \begin{array}{cc} a & b \\ c & d  \end{array} \right) \in
\mathrm{GL}_2(\mathcal{O}_{\mathcal{L}}), \ \ c \equiv 0(\mathrm{mod}\: \pi\mathcal{O}_{\mathcal{L}})\right\}
$$
and $w = [W]$, with $W = \left(  \begin{array}{cc} 0 & 1 \\ \pi & 1   \end{array}\right)$.

Let $T$ be the ``diagonal'' maximal $\mathcal{K}$-torus in $\mathrm{PGL}_2/\mathcal{K}$, and $\overline{T}$ be the ``diagonal" $k$-torus in $\mathrm{PGL}_2/k$.  Set $N = T \rtimes \langle w \rangle$ and $\overline{N} = \overline{T} \rtimes \langle \omega \rangle$, where $\omega$ is the automorphism of $\overline{T}$ of order two defined by $\omega(t) = t^{-1}$ for $t \in \overline{T}$.

For an element $a \in \mathcal{O}_{\mathcal{L}}$, we denote by $\bar{a}$ its residue $a(\mathrm{mod}\: \pi\mathcal{O}_{\mathcal{L}}) \in \ell$. Then for $A = \left( \begin{array}{cc} a & b \\ c & d \end{array}  \right) \in \mathrm{GL}_2(\mathcal{O}_{\mathcal{L}})$, sending $[A]$ to
$\left[ \left( \begin{array}{cc} \bar{a} & 0 \\ 0 & \bar{d} \end{array}  \right) \right]$ and $w$ to $\omega$ defines a homomorphism of $\mathcal{G}$-groups
$$
\lambda \colon \Sigma_{\mathcal{L}}(\mathcal{M}) \to \overline{N}(\ell)
$$
with kernel
$$
\Sigma_0 = \left\{ [A] \ \colon \ A = \left( \begin{array}{cc} a & b \\ c & d \end{array} \right) \in \mathrm{GL}_2(\mathcal{O}_{\mathcal{L}}), \ \ a,d \equiv 1(\mathrm{mod}\: \pi\mathcal{O}_{\mathcal{L}}), \ \ c \equiv 0(\mathrm{mod}\: \pi \mathcal{O}_{\mathcal{L}})     \right\}.
$$
The group $\Sigma_0$ has a filtration by normal subgroups of $\Sigma$:
$$
\Sigma_n = \left\{ [A] \ \colon \ A = \left( \begin{array}{cc} a & b \\ c & d \end{array} \right) \in \mathrm{GL}_2(\mathcal{O}_{\mathcal{L}}), \ \ a,d \equiv 1(\mathrm{mod}\: \pi^{n+1}\mathcal{O}_{\mathcal{L}}), \ \ b \equiv 0(\mathrm{mod}\: \pi^n\mathcal{O}_{\mathcal{L}})  \ \ c \equiv 0(\mathrm{mod}\: \pi^{n+1} \mathcal{O}_{\mathcal{L}})     \right\}
$$
$(n \geq 1)$. The successive quotients $\Sigma_n/\Sigma_{n+1}$ are naturally 3-dimensional $\ell$-vector spaces. Furthermore, the group $\Sigma$ acts on each quotient by $\ell$-linear transformations.
Hence, for any cocycle $\zeta \in Z^1(\mathcal{G} , \Sigma)$, we have
$$
H^1(\mathcal{G} , {}_{\zeta}\Sigma_n / {}_{\zeta}\Sigma_{n+1}) = 1 \ \ \text{for all} \ \ n \geq 0
$$
(cf. \cite[Corollary 4.2.7]{GS}), implying that $H^1(\mathcal{G} , {}_{\zeta}\Sigma_0) = 1$. It follows that for the map $$\lambda^1 \colon H^1(\mathcal{G} , \Sigma_{\mathcal{L}}(\mathcal{M})) \to H^1(\mathcal{G} , \overline{N}(\ell))$$ induced by $\lambda$ and any $\zeta \in H^1(\mathcal{G} , \Sigma(\mathcal{M}))$, the fiber $(\lambda^1)^{-1}(\lambda^1(\zeta))$ consists of a single element, proving that $\lambda^1$ is injective. On the other hand, by Hensel's Lemma (or by direct computation --- see below), the natural reduction map
\begin{equation}\label{E:RedA001}
H^1(\mathcal{L}/\mathcal{K} , N(\mathcal{O}_{\mathcal{L}})) \longrightarrow H^1(\ell/k , \overline{N}(\ell))
\end{equation}
is a bijection. Consequently, the map $H^1(\mathcal{L}/\mathcal{K} , N(\mathcal{O}_{\mathcal{L}})) \to H^1(\mathcal{L}/\mathcal{K} , \Sigma(\mathcal{M}))$ is also
a bijection, and hence $\mathcal{H}_1$ is the image of the map $H^1(\mathcal{L}/\mathcal{K} , N(\mathcal{O}_{\mathcal{L}})) \to H^1(\mathcal{L}/\mathcal{K} , G)$. So, to identify $\mathcal{H}_1$ we need to compute $H^1(\mathcal{L}/\mathcal{K} , N(\mathcal{O}_{\mathcal{L}}))$.

\vskip1mm

We have an exact sequence
$$
H^1(\mathcal{L}/\mathcal{K} , T(\mathcal{O}_{\mathcal{L}})) \longrightarrow H^1(\mathcal{L}/\mathcal{K} , N(\mathcal{O}_{\mathcal{L}})) \stackrel{\nu}{\longrightarrow} H^1(\mathcal{L}/\mathcal{K} , \langle w \rangle)
$$
of pointed sets.
Since $T \simeq \mathbb{G}_m$ and $\mathcal{L}/\mathcal{K}$ is unramified, we have $H^1(\mathcal{L}/\mathcal{K} , T(\mathcal{O}_{\mathcal{L}})) = 1$ (see, e.g. \cite[Chapter IV, Proposition 1.1]{MilneCFT}). Therefore,
the fiber $\nu^{-1}(e)$, where $e \in H^1(\mathcal{L}/\mathcal{K} , \langle w \rangle)$ is the trivial element, consists of the trivial class. Thus, $H^1(\mathcal{L}/\mathcal{K} , N(\mathcal{O}_{\mathcal{L}})) \setminus \{ 1 \}$ is precisely the fiber $\nu^{-1}(\nu(\varepsilon))$, where $\varepsilon \colon \mathcal{G} \to N(\mathcal{O}_{\mathcal{L}})$ is a cocycle defined by $\varepsilon(\sigma) = w$ for the nontrivial element $\sigma \in \mathcal{G}$. By
\cite[Ch. I, \S 5, Proposition 39]{Serre-GC},
there is a bijection
$$
H^1(\mathcal{L}/\mathcal{K} , {}_{\varepsilon} T(\mathcal{O}_{\mathcal{L}})) \longrightarrow \nu^{-1}(\nu(\varepsilon))
$$
that, on the level of cocycles, is given by
multiplication by $\varepsilon$: $\xi \mapsto \xi \varepsilon$. It is well known that the twisted torus ${}_{\varepsilon} T$ is isomorphic to the norm 1 torus $\mathrm{R}^{(1)}_{\mathcal{L}/\mathcal{K}}(\mathbb{G}_m)$ (see \cite[\S2.2, Example 1]{PlRR}). Consequently, we have $H^1(\mathcal{L}/\mathcal{K} , {}_{\varepsilon} T) \simeq \mathcal{K}^{\times}/N_{\mathcal{L}/\mathcal{K}}(\mathcal{L}^{\times})$ (see \cite[Lemma 2.22]{PlRR}), and similarly,
\begin{equation}\label{E:IsomA001}
H^1(\mathcal{L}/\mathcal{K} , {}_{\varepsilon}T(\mathcal{O}_{\mathcal{L}})) \simeq \mathcal{O}^{\times} / N_{\mathcal{L}/\mathcal{K}}(\mathcal{O}_{\mathcal{L}}^{\times}).
\end{equation}
(The same computation shows that if $\bar{\varepsilon} \colon \overline{\mathcal{G}} = \mathrm{Gal}(\ell/k) \to \langle \omega \rangle$ is the nontrivial homomorphism, viewed as a cocycle, then
$$
H^1(\ell/k , {}_{\bar{\varepsilon}} \overline{T}) \simeq k^{\times}/N_{\ell/k}(\ell^{\times}).
$$
It follows that the reduction map $H^1(\mathcal{L}/\mathcal{K} , {}_{\varepsilon} T(\mathcal{O}_{\mathcal{L}})) \to H^1(\ell/k , {}_{\bar{\varepsilon}} \overline{T})$ is surjective, so (\ref{E:RedA001}) is also surjective.)

Now, it is straightforward to check by direct computation that the isomorphism (\ref{E:IsomA001}) can be described explicitly as follows: given $a \in \mathcal{O}^{\times}$, the map $\zeta_a \colon \mathcal{G} \to {}_{\varepsilon}T(\mathcal{O}_{\mathcal{L}})$ defined by
$$
\zeta_a(1) = [I_2] \ \ \text{and} \ \ \zeta_a(\sigma) = \left[ \left( \begin{array}{cc} a & 0 \\ 0 & 1 \end{array}    \right) \right]
$$
is a  cocycle, and the correspondence $a \mapsto \zeta_a$ leads the isomorphism inverse to (\ref{E:IsomA001}).

Thus, every class in $H^1(\mathcal{L}/\mathcal{K} , N(\mathcal{O}_{\mathcal{L}}))$ is represented by a cocycle of the form $\zeta = \zeta_a \varepsilon$ that is defined by
$$
(\zeta_a \varepsilon)(\sigma) =  \left[ \left(  \begin{array}{cc} a & 0 \\ 0 & 1  \end{array} \right) \right] \cdot  \left[ \left(  \begin{array}{cc} 0 & 1 \\ \pi & 0   \end{array} \right) \right] =  \left[ \left(  \begin{array}{cc} 0 & a \\ \pi & 0  \end{array} \right) \right].
$$
Set $A = \left(  \begin{array}{cc} 0 & a \\ \pi & 0  \end{array} \right) $ and note that $A^2 = (\pi a)I_2$. The twisted algebra $\mathcal{D} = {}_{\zeta} \mathrm{M}_2(\mathcal{K})$ is then described as
$$
\mathcal{D} = \left\{ X \in \mathrm{M}_2(\mathcal{L}) \ \colon \ (\mathrm{Int}\: A)(\sigma(X)) = X      \right\}
$$
for the nontrivial element $\sigma \in \mathcal{G}$. The set $\tilde{\mathcal{L}}$ of diagonal matrices in $\mathcal{D}$ is a $\mathcal{K}$-subalgebra isomorphic to $\mathcal{L}$; furthermore, $A$ clearly lies in $\mathcal{D}$, with $\mathrm{Int}\: A$ inducing on $\tilde{\mathcal{L}}$ the nontrivial automorphism.
It follows that $\mathcal{D}$ is the quaternion algebra $(\mathcal{L}/\mathcal{K} , \pi a)$ obtained by the standard construction of cyclic algebras using the quadratic extension $\mathcal{L}/\mathcal{K}$ and the element $\pi a \in \mathcal{K}^{\times}$ (cf. \cite[\S2.5]{GS}).

To summarize, we have the following bijections
$$
\mathcal{O}^{\times} / N_{\mathcal{L}/\mathcal{K}}(\mathcal{O}_{\mathcal{L}}^{\times}) \simeq H^1(\mathcal{L}/\mathcal{K} , {}_{\varepsilon}T(\mathcal{O}_{\mathcal{L}})) \longrightarrow \nu^{-1}(\nu(\varepsilon)) = H^1(\mathcal{L}/\mathcal{K} , N(\mathcal{O}_{\mathcal{L}})) \setminus \{ 1 \},
$$
and the composition of these maps with $H^1(\mathcal{L}/\mathcal{K} ,  N(\mathcal{O}_{\mathcal{L}})) \to H^1(\mathcal{L}/\mathcal{K} , G)$ followed by $\kappa_{\mathcal{L}/\mathcal{K}}$ is given by
$$
a N_{\mathcal{L}/\mathcal{K}}(\mathcal{O}_{\mathcal{L}}^{\times}) \mapsto [(\mathcal{L}/\mathcal{K} , \pi a)],
$$
where $[ * ]$ denotes the corresponding isomorphism class of the quaternion algebra. It follows that $\kappa_{\mathcal{L}/\mathcal{K}}$ yields a bijection between $\mathcal{H}_1 \setminus \{ 1 \}$ and $Q(\mathcal{L}/\mathcal{K})_{\mathrm{ram}}$, as claimed.

\section{Action on the product of two trees}\label{S:Product}

In this section, we will make preparations for proving Theorems \ref{T:Main-sep} and \ref{T:Main-quadratic} in the next section by suitably adapting the approach described in \S \ref{S:BT} in the local situation. This adaptation also builds on the proof of the Raghunathan-Ramanathan Theorem given in \cite{ARR}.

Suppose $k$ is a field of characteristic $\neq 2$. Let $K= k(x)$ be the field of rational functions in one variable over $k$ and $A = k[x , x^{-1}]$ be the corresponding ring of Laurent polynomials. We let $v$ and $v^-$ be the discrete valuations of $K$ associated with $x$ and $x^{-1}$; the corresponding completions are then $K_v = k((x))$ and $K_{v^-} = k((x^{-1}))$ (fields of formal Laurent series), with the valuation rings $O := k[[x]]$ and $O^- = k[[x^{-1}]]$ (rings of formal power series). We note that
$$
A \bigcap O = k[x] \ \ \text{and} \ \ A \bigcap O^- = k[x^{-1}].
$$
Furthermore, let $\mathcal{T}$ and $\mathcal{T}^{-}$ be the Bruhat-Tits trees associated with $k((x))$ and $k((x^{-1}))$ --- see \S \ref{S:BT}. The goal of this section is to obtain information about the (diagonal) action of the group $\Gamma = \mathrm{PGL}_2(A)$ on the product $\mathcal{X} = \mathcal{T} \times \mathcal{T}^-$. More precisely, we will describe a fundamental domain for this action and identify the stabilizers of points in this fundamental domain. The reader should be advised that these results will be applied in the next section over a given quadratic extension $\ell$ of $k$  (i.e. to the field $L = \ell(x)$, the group $\Gamma_{\ell} = \mathrm{PGL}_2(\ell[x , x^{-1}])$, etc.), but in order to keep our notations simple, we will formulate the results of this section over $k$.

\subsection{Fundamental domain} To construct a fundamental domain for the action of $\Gamma = \mathrm{PGL}_2(A)$, where $A = k[x , x^{-1}]$, on $\mathcal{T} \times \mathcal{T}^-$, we consider a 2-dimensional vector space $W_0$ over $k$ with a fixed
basis $e_1 , e_2$. We will also use this basis as the standard basis in the vector spaces $W = W_0 \otimes_k k((x))$ and $W^- = W_0 \otimes_k k((x^{-1}))$ involved in the construction of $\mathcal{T}$ and $\mathcal{T}^-$ (cf. \S \ref{S:BT}). For $n \in \mathbb{Z}$, we consider the $O$-lattice
$$
\Lambda_n = Oe_1 + Ox^ne_2
$$
in $W$, and let $P_n = [\Lambda_n]$ denote the corresponding vertex in $\mathcal{T}$.
Similarly, we consider the $O^-$-lattices
$$
\Lambda^-_n = O^-e_1 + O^- (x^{-1})^n e_2 \ \ (n \in \mathbb{Z})
$$
in $W^-$, and the corresponding points $P^-_n = [\Lambda^-_n] \in \mathcal{T}^-$. We take $\mathcal{E}$ to be the edge $P_0P_1$ in $\mathcal{T}$, and let $\mathcal{A}^{-}$ denote the line (apartment) in $\mathcal{T}^-$ spanned by the vertices $\ldots, P^{-}_{-1}, P^-_0, P^-_1, \ldots$ 
(we note that $\mathcal{A}^-$ is the union of the edges $\mathcal{E}^-_n = P^-_n P^-_{n+1}$ for all $n \in \mathbb{Z}$).
Set $\Phi = \mathcal{E} \times \mathcal{A}^-$.

\begin{prop}\label{P:FD}
With the preceding notations, we have
$\Gamma \cdot \Phi = \mathcal{T} \times \mathcal{T}^-$, i.e. $\Phi$ is a \emph{weak} fundamental domain for the action of $\Gamma$ on $\mathcal{T} \times \mathcal{T}^-$.
\end{prop}

This is a formal consequence of \cite[Proposition 3(1)]{Abr} (see also \cite[Proposition 5]{Abr}),  but we would like to provide additional details and to put this result in a more general context. Let $\mathbb{G}$ be a simply connected Chevalley group over $k$, let $\mathbb{T}$, $\mathbb{N}$, and $\mathbb{B}$ and $ \mathbb{B}_-$ be a maximal $k$-split torus, its normalizer, and a pair of opposite Borel subgroups containing $\mathbb{T}$. We consider the group homomorphisms
$$
\rho \colon \mathbb{G}(k[x]) \to \mathbb{G}(k) \ \ \text{and} \ \ \rho^- \colon \mathbb{G}(k[x^{-1}]) \to \mathbb{G}(k)
$$
induced by the $k$-algebra homomorphisms $k[x] \to k$ and $k[x^{-1}] \to k$ sending $x$ (resp., $x^{-1}$) to $0$. Set $B = \rho^{-1}(\mathbb{B}(k))$, $B^- = (\rho^-)^{-1}(\mathbb{B}_-(k))$ and $N = \mathbb{N}(A)$. Then the group $\mathbb{G}(A)$ has two $BN$-pairs $(B , N)$ and $(B^- , N)$ --- cf. \cite{MT}, \cite{M}. In fact, by \cite[Lemma 5]{Abr}, $(B, B^-, N)$ is a twin $BN$-pair in $\mathbb{G}(A)$.
(We refer the reader to \cite[\S6.3]{AbBr} for a detailed discussion of these concepts.) The main consequence of this needed for the proof of Proposition \ref{P:FD} is the following version of the Birkhoff decomposition:
\begin{equation}\label{E:BD1}
\mathbb{G}(A) = B N B^-.
\end{equation}
(see \cite[Proposition 6.81]{AbBr} or \cite[\S 3.2]{Tits}).

We now specialize this result to $G = \mathrm{PGL}_2$. Using (\ref{E:BD1}) for $\mathrm{SL}_2$, one easily derives the following decomposition for $\Gamma = G(A)$:
\begin{equation}\label{E:BD2}
\Gamma = B N B^-
\end{equation}
where
$$
B = \{ [A] \ \vert \ A = \left(\begin{array}{cc} a & b \\ c & d \end{array}   \right) \in \mathrm{GL}_2(k[x]) \ \ \text{such that} \ \ c \equiv 0 (\mathrm{mod}\: xk[x]) \},
$$
$$
B^- = \{ [A] \ \vert \ A = \left(\begin{array}{cc} a & b \\ c & d \end{array}   \right) \in \mathrm{GL}_2(k[x^{-1}]) \ \ \text{such that} \ \ b \equiv 0 (\mathrm{mod}\: x^{-1}k[x^{-1}]) \},
$$
and $N = T \cup sT$, with $s = \left[ \begin{array}{cc} 0 & 1 \\ 1 & 0 \end{array} \right]$ and
$$
T = \{ [A] \ \vert \ A \in \mathrm{GL}_2(A) \ \ \text{is a diagonal  matrix} \}.
$$
We have
$$
B \mathcal{E} = \mathcal{E} \ \ \text{and} \  B^- \mathcal{E}^-_{-1} = \mathcal{E}^-_{-1}.
$$
(see Lemmas \ref{L:Stab1} and \ref{L:Stab2} below). On the other hand, it is easy to see that the closure of $\Gamma$ in $G(k((x)))$ contains $T \cdot \mathrm{PSL}_2(k((x)))$, implying that
$$
\Gamma \mathcal{E} = (T \cdot \mathrm{PSL}_2(k((x)))) \mathcal{E} = \mathcal{T},
$$
and similarly, $\Gamma \mathcal{E}^-_{-1} = \mathcal{T}^-$. Since $N \mathcal{E}^-_{-1} = \mathcal{A}^-$, it follows from (\ref{E:BD2}) that
$$
\mathcal{T}^- = \Gamma \mathcal{E}^-_{-1} = B \mathcal{A}^-.
$$
Now, let $(P , P^-) \in \mathcal{T} \times \mathcal{T}^-$. We can find $\gamma \in \Gamma$ so that $\gamma^{-1} P \in \mathcal{E}$. Furthermore, there exists $\beta \in B$ such that $\beta^{-1}(\gamma^{-1} P^-) \in \mathcal{A}^-$. Noting that $\beta^{-1}(\gamma^{-1} P) = \gamma^{-1} P$, we obtain
$$
(\gamma \beta)^{-1} (P , P^-) = ((\gamma^{-1} P) , \beta^{-1}(\gamma^{-1} P^-)) \in \Phi,
$$
proving Proposition \ref{P:FD}. \hfill $\Box$

\begin{remark}
Here is another construction of a fundamental domain. Let $\mathcal{S}$ be the union of all edges in $\mathcal{T}$ emanating from $P_0$ (the ``star'' of $P_0$), and let $\mathcal{A}^-_+$ be the ray in $\mathcal{T}^-$ spanned by the vertices $P_0^-, P_1^-, \ldots $. Then $\Psi = \mathcal{S} \times \mathcal{A}^-_+$ is a weak fundamental domain for the action of $\Gamma$ on $\mathcal{T} \times \mathcal{T}^-$, i.e. $\Gamma \cdot \Psi = \mathcal{T} \times \mathcal{T}^-$.
This is an easy consequence of the following facts:
\begin{itemize}

\item $\Gamma \mathcal{E} = \mathcal{T}$;

\item $\mathcal{S} = G(k[x]) \mathcal{E}$;

\item $G(k[x]) \mathcal{A}^-_+ = \mathcal{T}^-$.
\end{itemize}
The first fact was mentioned above in the proof of Proposition \ref{P:FD}, the second one is immediate, and the third one is well-known and is established,
for example, in \cite[Ch. II, \S 1.6]{Serre-Trees} (we note that it is also a consequence of a certain ``Birkhoff decomposition"). The advantage of the fundamental domain $\Phi$ over $\Psi$ is that the stabilizers of points of $\Phi$ are easier to compute --- we will see these computations in the next subsection. On the other hand, while the fundamental domain $\Phi$ is specific to the group of points over the ring of Laurent polynomials, fundamental domains similar to $\Psi$ can be constructed over more general localizations of the polynomial ring.

\end{remark}

\vskip1mm

\subsection{Stabilizers} Given a point $\mathcal{P} = (P , P^-) \in \mathcal{T} \times \mathcal{T}^-$, its stabilizer for the action of $\Gamma$ will be denoted $\Sigma(\mathcal{P})$ or $\Sigma(P , P^-)$. The goal of this subsection is to identify the stabilizers $\Sigma(\mathcal{P})$ of points $\mathcal{P}$  in the fundamental domain $\Phi = \mathcal{E} \times \mathcal{A}^-  \subset \mathcal{T} \times \mathcal{T}^-$ constructed in Proposition \ref{P:FD}. It turns out that we have the following.
\begin{prop}\label{P:Stab-FD}
 For $\mathcal{P} \in \Phi$, the stabilizer $\Sigma(\mathcal{P})$ is one of the following groups:

\begin{itemize}

\item[{\rm (i)}] $\left\{ [A] \ \vert \ A = \left( \begin{array}{cc} a & b \\ 0 & d \end{array} \right) \ \ \text{with} \ \ a , d \in k^{\times}, \ b \in \mathcal{V} \right\}$ or  $\left\{ [A] \ \vert \ A = \left( \begin{array}{cc} a & 0 \\ c & d \end{array} \right) \ \ \text{with} \ \  a , d \in k^{\times}, \ c  \in \mathcal{V} \right\},$ where $\mathcal{V} \subset k[x , x^{-1}]$ is a finite-dimensional $k$-vector subspace that has a basis consisting of powers of $x$;

\vskip1mm

\item[{\rm (ii)}] $\mathrm{PGL}_2(k)$;

\vskip1mm

\item[{\rm (iii)}] $[t] \cdot \mathrm{PGL}_2(k) \cdot [t]^{-1}$,  \  where \ $t = \mathrm{diag}(1 , x)$;

\vskip1mm

\item[{\rm (iv)}] $\left\{ [A] \ \vert \ A = \left(  \begin{array}{cc} a & 0 \\ 0 & d \end{array}    \right) \ \ \text{with} \ \ a,d \in k^{\times}  \right\} \rtimes \langle [g] \rangle,$ where $g = \left( \begin{array}{cc} 0 & 1 \\ x & 0 \end{array} \right)$.
\end{itemize}
\end{prop}

The proof of the proposition will reveal which points $\mathcal{P} \in \Phi$ yield each type of stabilizer.

\vskip5mm

As above, for $n \in \mathbb{Z}$, we consider the $O$-lattice
$$
\Lambda_n = O e_1 + O x^n e_2
$$
in $W$ with corresponding vertex $P_n = [\Lambda_n] \in \mathcal{T}$, and similarly, the $O^-$-lattice
$$
\Lambda^-_n = O^- e_1 + O^-(x^{-1})^n e_2
$$
in $W^-$ with corresponding vertex $P^-_n = [\Lambda^-_n] \in \mathcal{T}^-$. Clearly, the stabilizer of $P_0$ is
\begin{equation}\label{E:Stab1}
\mathrm{Stab}_{\Gamma}(P_0) = \{ [A] \ \vert \ A \in \mathrm{GL}_2(k[x]) \}.
\end{equation}
Set
$$
t = \mathrm{diag}(1 , x) \ \ \text{and} \ \ g = \left( \begin{array}{cc} 0 & 1 \\ x & 0 \end{array}   \right) \ \ \ \in \ \Gamma.
$$
Then $t(\Lambda_0) = \Lambda_1$, hence $[t](P_0) = P_1$, and
$$
g(\Lambda_0) = O x e_2 + O e_1 = \Lambda_1 \ \ \text{and} \ \ g(\Lambda_1) = O x e_2 + O x_1 e_1 = x \Lambda_0,
$$
hence $[g](P_i) = P_{1 - i}$ for $i = 0, 1$. Consequently,
\begin{equation}\label{E:Stab2}
\mathrm{Stab}_{\Gamma}(P_1) = [t] \cdot \mathrm{Stab}_{\Gamma}(P_0) \cdot [t]^{-1} = \left\{ [A] \ \vert \ A = \left( \begin{array}{cc} a & x^{-1} b \\ x c & d  \end{array}  \right) \ \ \text{with} \ \ \left( \begin{array}{cc} a & b \\ c & d  \end{array}  \right) \in \mathrm{GL}_2(k[x])  \right\}
\end{equation}
It follows that
\begin{equation}\label{E:Stab3}
F := \mathrm{Stab}_{\Gamma}(P_0) \bigcap \mathrm{Stab}_{\Gamma}(P_1) = \{ [A] \ \vert \ A = \left( \begin{array}{cc} a & b \\ c & d \end{array} \right) \in \mathrm{GL}_2(k[x]) \ \ \text{such that} \ \ c \equiv 0 (\mathrm{mod}\: \ xk[x]) \}.
\end{equation}
Let $M$ be the midpoint of $\mathcal{E} = P_0P_1$. Then for any $P \in \mathcal{E} \setminus \{ M, P_0, P_1 \}$, the stabilizer is $\mathrm{Stab}_{\Gamma}(P) = F$. If $h \in \mathrm{Stab}_{\Gamma}(M)$, then either $h \in F$ or $h(P_i) = P_{1-i}$ for $i = 0, 1$, and then $h \in F \cdot [g]$. Since $[g]^2 = 1$, we see that $\mathrm{Stab}_{\Gamma}(M) = F \rtimes \langle [g] \rangle$. Thus, we have proved the following.

\begin{lemma}\label{L:Stab1}
The stabilizers of the points $P_0$ and $P_1$ for the action of $\Gamma$ on $\mathcal{T}$ are given by equations (\ref{E:Stab1}) and (\ref{E:Stab2}), respectively. Let $M$ be the midpoint of $\mathcal{E}$. Then the stabilizer of any point $P \in \mathcal{E} \setminus \{M, P_0, P_1\}$ is $F$ given by (\ref{E:Stab3}). Finally, $\mathrm{Stab}_{\Gamma}(M) = F \rtimes \langle [g] \rangle$. \end{lemma}

Next, we will establish a similar, but more general, result for $\mathcal{T}^-$. Again, we have
$$
\mathrm{Stab}_{\Gamma}(\Lambda^-_0) = \mathrm{PGL}_2(k[x^{-1}]).
$$
Set
$$
t^-_n = \mathrm{diag}(1 , x^{-n}) \ \ \text{and} \ \ g^-_n = \left( \begin{array}{cc} 0 & 1 \\ x^{-(2n+1)} & 0 \end{array} \right).
$$
Then $t^-_n(\Lambda^-_0) = \Lambda^-_n$, hence $[t^-_n](P^-_0) = P^-_n$, and
$$
g^-_n(\Lambda^-_n) = O^- x^{-(2n+1)} e_2 + O^- x^{-n} e_1 = x^{-n} \Lambda^-_{n+1}
$$
and
$$
g^-_n(\Lambda^-_{n+1}) = O^- x^{-(2n+1)} e_2 + O^- x^{-(n+1)} e_1 = x^{-(n+1)} \Lambda^-_n,
$$
hence $[g^-_n](P^-_i) = P^-_{(2n+1) - i}$ for $i = n, n+1$. So,
\begin{equation}\label{E:Stab4}
\mathrm{Stab}_{\Gamma}(P^-_n) = [t^-_n] \cdot \mathrm{Stab}_{\Gamma}(P^-_0) \cdot [t^-_n]^{-1} =
\end{equation}
$$
\left\{ [A] \ \vert \ A = \left( \begin{array}{cc} a & x^n b \\ x^{-n} c & d \end{array} \right) \ \ \text{with} \ \ \left( \begin{array}{cc} a & b \\ c & d \end{array} \right) \in \mathrm{GL}_2(k[x^{-1}])   \right\}.
$$
Since $x^n k[x^{-1}] \subset x^{n+1} k[x^{-1}]$ and $x^{-n} k[x^{-1}] \supset x^{-(n+1)} k[x^{-1}]$ for all $n \in \mathbb{Z}$, it follows that
\begin{equation}\label{E:Stab5}
F^-_n := \mathrm{Stab}_{\Gamma}(P^-_n) \bigcap \mathrm{Stab}_{\Gamma}(P^-_{n+1}) =
\end{equation}
$$
\left\{ [A] \ \vert \ A = \left( \begin{array}{cc}  a & b \\ c & d    \end{array} \right) \ \ \text{with} \ \ a,d \in k[x^{-1}], \ b \in x^n k[x^{-1}], \ c \in x^{-(n+1)} k[x^{-1}] \ \ \text{and} \ \ \det A \in k^{\times}    \right\}
$$
Let $M^-_n$ be the midpoint of $\mathcal{E}^-_n = P^-_n P^-_{n+1}$. Then for any point $P \in \mathcal{E}^-_n \setminus \{ M^-_n, P^-_n, P^-_{n+1} \}$, the stabilizer is $\mathrm{Stab}_{\Gamma}(P) = F^-_n$. On the other hand, if $h \in \mathrm{Stab}_{\Gamma}(M^-_n)$, then either $h \in F^-_n$ or $h(P^-_i) = P_{(2n+1) - i}$ for $i = n, n+1$. It follows that
$$
\mathrm{Stab}_{\Gamma}(M^-_n) = F^-_n \rtimes \langle [g^-_n] \rangle.
$$

\begin{lemma}\label{L:Stab2}
For any $n \in \mathbb{Z}$, the stabilizer of $P^-_n$ is given by the equation (\ref{E:Stab4}).  Let $M^-_n$ be the midpoint of $\mathcal{E}^- = P^-_n P^-_{n+1}$.  Then the stabilizer of any point $P \in \mathcal{E}^-_n \setminus \{ M^-_n , P^-_n, P^-_{n+1} \}$ is $F^-_n$ given by (\ref{E:Stab5}). Finally,  $\mathrm{Stab}_{\Gamma}(M^-_n) = F^-_n \rtimes \langle [g^-_n] \rangle$.
\end{lemma}

\vskip3mm

\noindent {\it Proof of Proposition \ref{P:Stab-FD}.} Our goal is to identify the stabilizer $\Sigma(\mathcal{P}) = \mathrm{Stab}_{\Gamma}(\mathcal{P})$ of an arbitrary point $\mathcal{P} = (P , P^-) \in \Phi$. Clearly,
$$
\Sigma(\mathcal{P}) = \mathrm{Stab}_{\Gamma}(P) \bigcap \mathrm{Stab}_{\Gamma}(P^-),
$$
so we can use Lemmas \ref{L:Stab1} and \ref{L:Stab2}. We will consider separately the cases where $P = P_0$, $P_1$, $P \in \mathcal{E} \setminus \{P_0, P_1, M\}$, and $P = M$.

\vskip2mm

\noindent \underline{\it Case 1: $P = P_0$.} First, let $P^- = P^-_n$. Comparing (\ref{E:Stab1}) and (\ref{E:Stab4}), we see that for $n > 0$, we have
$$
\Sigma(\mathcal{P}) = \left\{ [A] \ \vert \ A = \left( \begin{array}{cc} a & b \\ 0 & d \end{array}  \right) \in \mathrm{GL}_2(k[x]) \ \ \text{with} \ \ a,d \in k^{\times} \ \ \text{and} \ \ \deg b \leq n \right\},
$$
for $n = 0$, we have $\Sigma(\mathcal{P}) = \mathrm{PGL}_2(k)$, and for $n < 0$, we have
$$
\Sigma(\mathcal{P}) = \left\{ [A] \ \vert \ A = \left( \begin{array}{cc} a & 0 \\ c & d \end{array} \right) \in \mathrm{GL}_2(k[x]) \ \ \text{with} \ \ a,d \in k^{\times}  \ \ \text{and} \ \ \deg c \leq |n|  \right\},
$$
proving the proposition in this case.

\vskip.5mm

$\bullet$ Let $P^- \in \mathcal{E}^-_n \setminus \{ P^-_n, P^-_{n+1}, M^-_n \}$. Using (\ref{E:Stab1}) and (\ref{E:Stab5}), we see that for $n \geq 0$, we have
$$
\Sigma(\mathcal{P}) = \left\{ [A] \ \vert \ A = \left( \begin{array}{cc} a & b \\ 0 & d \end{array}  \right) \in \mathrm{GL}_2(k[x]) \ \ \text{with} \ \ a,d \in k^{\times} \ \ \text{and} \ \ \deg b \leq n \right\},
$$
(same as in the previous case),
and for $n < 0$, we have
$$
\Sigma(\mathcal{P}) = \left\{ [A] \ \vert \ A = \left( \begin{array}{cc} a & 0 \\ c & d \end{array}  \right) \in \mathrm{GL}_2(k[x]) \ \ \text{with} \ \ a,d \in k^{\times} \ \ \text{and}\ \ \deg c \leq |n| - 1  \right\}.
$$

\vskip.5mm

$\bullet$ $P^- = M^-_n$. By considering the determinant, we see that
$$
\mathrm{GL}_2(k[x]) \bigcap \left( F^-_n \rtimes \langle [g^-_n] \rangle  \right) = \mathrm{GL}_2(k[x]) \bigcap F^-_n,
$$
and hence the stabilizers in this case are exactly the same as in the previous case. This completes the proof of the proposition in Case 1.

\vskip2mm

\noindent \underline{\it Case 2: $P = P_1$.} Again, we start with $P^-= P^-_n$. Comparing (\ref{E:Stab2}) and (\ref{E:Stab4}), we see that for $n \geq 0$, we have
$$
\Sigma(\mathcal{P}) = \left\{ [A] \ \vert \ A = \left( \begin{array}{cc} a & b \\  0 & d \end{array} \right) \ \ \text{with} \ \ a,d \in k^{\times} \ \ \text{and} \ \ b = x^{-1}f(x) \ \ \text{where} \ \ \deg f \leq n+1    \right\},
$$
for $n = -1$, we have
$$
\Sigma(\mathcal{P}) = \left\{ [A] \ \vert \ A = \left( \begin{array}{cc} a & b \\ c & d \end{array}  \right) \ \ \text{with} \ \ a,d \in k, \ b = \beta x^{-1}, \ c = \gamma x \ \ \text{for some} \ \ \beta , \gamma \in k, \ \ \det A \in k^{\times}   \right\} =
$$
$$
[t] \cdot \mathrm{PGL}_2(k) \cdot [t]^{-1} \ \ \text{where} \ \ t = \mathrm{diag}(1 , x),
$$
and for $n < -1$, we have
$$
\Sigma(\mathcal{P}) = \left\{ [A] \ \vert \ A = \left( \begin{array}{cc} a & 0 \\ c & d \end{array}  \right) \ \ \text{with} \ \ a , d \in k^{\times} \ \ \text{and} \ \ c = xf(x) \ \ \text{where} \ \ \deg f \leq |n| - 1  \right\}.
$$

\vskip.5mm

$\bullet$ $P^- \in \mathcal{E}^-_n \setminus \{ P^-_n, P^-_{n+1}, M^-_n\}$. Using (\ref{E:Stab2}) and (\ref{E:Stab5}) we see that for $n \geq -1$, we have
$$
\Sigma(\mathcal{P}) = \left\{ [A] \ \vert \ A = \left( \begin{array}{cc} a & b \\ 0 & d \end{array} \right) \ \ \text{with} \ \ a , d \in k^{\times}, \ b = x^{-1} f(x) \ \ \text{where} \ \ \deg f \leq n+1    \right\},
$$
and for $n < -1$, we have
$$
\Sigma(\mathcal{P}) = \left\{ [A] \ \vert \ A = \left( \begin{array}{cc} a & 0 \\ c & d \end{array} \right) \ \ \text{with} \ \ a , d \in k^{\times},  \ c = x f(x) \ \ \text{where} \ \
\deg f \leq |n| - 2 \right\}.
$$

\vskip.5mm

$\bullet$ $P^- = M^-_n$. Again, by considering the determinant, we conclude that
$$
\mathrm{Stab}_{\Gamma}(P_1) \bigcap (F^-_n \rtimes \langle [g^-_n] \rangle) = \mathrm{Stab}_{\Gamma}(P_1) \bigcap F^-_n ,
$$
so the stabilizers are exactly the same as in the previous case. This concludes the consideration of Case 2.

\vskip2mm

\noindent \underline{\it Case 3: $P \in \mathcal{E} \setminus \{ P_0, P_1, M \}$.} In this case, $\mathrm{Stab}_{\Gamma}(P) = \mathrm{Stab}_{\Gamma}(P_0) \cap \mathrm{Stab}_{\Gamma}(P_1) = F$, so the following results are obtained by takings the intersections of the corresponding results in Cases 1 and 2.

\vskip.5mm

$\bullet$ $P^- = P^-_n$. Then for $n \geq 0$, we have
$$
\Sigma(\mathcal{P}) = \left\{ [A] \ \vert \ A = \left( \begin{array}{cc} a & b \\ 0 & d \end{array}   \right) \ \ \text{with} \ \ a,d \in k^{\times} \ \ \text{and} \ \ b \in k[x], \ \text{with} \ \deg b \leq n    \right\},
$$
and for $n < 0$, we have
$$
\Sigma(\mathcal{P}) = \left\{ [A] \ \vert A= \left( \begin{array}{cc} a & 0 \\ c & d   \end{array} \right)  \ \ \text{with} \ \ a , d \in k^{\times}, \ c = xf(x) \ \ \text{where} \ \ f(x) \in k[x], \ \deg f \leq |n| - 1  \right\}.
$$

\vskip.5mm

$\bullet$ $P^- \in \mathcal{E}^-_n \setminus \{ P^-_n, P^-_{n+1}, M^-_n\}$. Then for $n \geq 0$, we have
$$
\Sigma(\mathcal{P}) = \left\{ [A] \ \vert \ A = \left( \begin{array}{cc} a & b \\ 0 & d \end{array} \right)  \ \ \text{with} \ \ a , d \in k^{\times} \ \ \text{and} \ \ b \in k[x], \ \deg b \leq n \right\},
$$
for $n = -1$, we have
$$
\Sigma(\mathcal{P}) = \left\{ [A] \ \vert \ A = \left( \begin{array}{cc} a & 0 \\ 0 & d  \end{array} \right) \ \ \text{with} \ \ a , d \in k^{\times}   \right\},
$$
and for $n < -1$,
$$
\Sigma(\mathcal{P}) = \left\{ [A] \ \vert \ A = \left( \begin{array}{cc} a  & 0 \\ c & d  \end{array}  \right)  \ \ \text{with} \ \ a , d \in k^{\times} \ \ \text{and} \ \ c = xf(x) \ \ \text{where} \ \ \deg f \leq | n | - 2  \right\}.
$$
Finally if $P^- = M^-_n$, the stabilizers are the same as in the previous case.

\vskip2mm

\noindent \underline{\it Case 4: $P = M$.} If $P^- \in \mathcal{E}^-_n \setminus \{ M^-_n \}$, then by considering the determinant, we see that
$$
\mathrm{Stab}_{\Gamma}(P) \bigcap \mathrm{Stab}_{\Gamma}(P^-) = (F \rtimes \langle [g] \rangle) \bigcap \mathrm{Stab}_{\Gamma}(P^-) = F \bigcap \mathrm{Stab}_{\Gamma}(P^-),
$$
so the stabilizers are the same as in Case 3.

Now suppose that $P^- = M^-_n$. Then
$$
\mathrm{Stab}_{\Gamma}(P) \bigcap \mathrm{Stab}_{\Gamma}(P^-) = \bigl(F \rtimes \langle [g]  \rangle \bigr) \bigcap \bigl( F^-_n \rtimes \langle [ g^-_n] \rangle \bigr) = \bigl( F \bigcap F^-_n  \bigr) \bigcup \bigl(F g \bigcap F^-_n g^-_n \bigr)
$$
by determinant considerations. Let us show that
\begin{equation}\label{E:Inters001}
   F g \bigcap F^-_n g^-_n = \varnothing \ \ \text{for} \ \ n \neq -1.
\end{equation}
Matrices in $F$ and $F^-_n$ have determinants in $k^{\times}$, and the only scalar multiple (up to a factor in $k^{\times}$) of $g (g^-_n)^{-1} = \mathrm{diag}(1 , x^{2n+2})$ having determinant in $k^{\times}$ is
$$
h = \mathrm{diag}(x^{-(n+1)} , x^{n+1}).
$$
Thus, if the intersection in (\ref{E:Inters001}) is nonempty, then there exists $s \in F$ such that $r := sh \in F^-_n$. Writing $s = \left( \begin{array}{cc}  a & b \\ c & d \end{array} \right) \in \mathrm{GL}_2(k[x])$ with $c \equiv 0(\mathrm{mod}\: xk[x])$, and
$$
r = \left( \begin{array}{cc} u & v \\ w & z \end{array}   \right) \ \ \text{with} \ \ u,z \in k[x^{-1}], \ v \in x^n k[x^{-1}], \ \ \text{and} \ \ w \in x^{-(n+1)} k[x^{-1}],
$$
we obtain
$$
u = a x^{-(n+1)}, \ v = b x^{n+1}, \ w = c x^{-(n+1)}, \  \ \text{and} \ \ z = d x^{n+1}.
$$
If $n > -1$, we obtain
$$
z \in xk[x] \cap k[x^{-1}] = \{ 0 \}  \ \ \text{and} \ \ v \in x^{n+1} k[x] \cap x^n k[x^{-1}] = \{ 0 \},
$$
which is impossible. Similarly, if $n < -1$, then
$$
u \in xk[x] \cap k[x^{-1}] = \{ 0 \} \ \ \text{and} \ \ w \in x^{-n} k[x] \cap x^{-(n+1)} k[x^{-1}] = \{ 0 \},
$$
which is also impossible. This proves (\ref{E:Inters001}).

So, if $n \neq -1$, we have $\Sigma(\mathcal{P}) = F \cap F^-_n$, which was already listed in Case 3. It remains to consider the case $n = -1$. Then $g = g^-_n$
$$
\Sigma(\mathcal{P}) = \bigl( F \cap F^-_n \bigr) \rtimes \langle [g] \rangle = \left\{ [A] \ \vert \ A = \left( \begin{array}{cc} a & 0 \\ 0 & d \end{array}  \right) \ \ \text{with} \ \ a , d \in k^{\times}   \right\} \rtimes \langle [g] \rangle.
$$
This completes the proof of Proposition \ref{P:Stab-FD}. \hfill $\Box$

\section{Prof of Theorem \ref{T:Main-quadratic}}\label{S:Proof1.2}

The goal of this section is to prove Theorem \ref{T:Main-quadratic}. Thus, we will show that for $G = \mathrm{PGL}_2$ and a quadratic extension $\ell/k$, the natural map
$$
\nu_{\ell} \colon H^1\bigl(\ell/k , G(\ell[x , x^{-1}])\bigr) \longrightarrow H^1\bigl( \ell((x))/k((x)) , G(\ell((x))) \bigr)
$$
is a bijection. The proof is based on an analysis of the action of the group $\Gamma_{\ell} = G(\ell[x , x^{-1}])$ extended by the Galois group $\mathrm{Gal}(\ell/k)$ on the product of the Bruhat-Tits trees associated with the fields $\ell((x))$ and $\ell((x^{-1}))$. Our argument relies heavily on the information developed in \S \ref{S:BT} concerning the fundamental domain and stabilizers for the action of $\Gamma_{\ell}$  on this space. As we already mentioned, actions of reductive algebraic groups over local fields on the associated buildings were used by Bruhat-Tits \cite{BT87} to compute Galois cohomology in the local situation. While we are going to use a largely similar approach in the global case, the spaces on which the relevant groups act in this situation are more general than buildings: in the present paper, we will use products of (two) trees. So, to prepare ourselves for the proof of Theorem \ref{T:Main-quadratic}, we will now record the needed
formalism in the general context of CAT(0) spaces. While
Lemma \ref{L:cohom} below is certainly well-known
to experts, to the best of our knowledge, it has not been previously
mentioned in the literature explicitly.

\subsection{Cohomology computations via fixed points} We refer the reader to \cite[\S11.1]{AbBr} for an overview of the theory of CAT(0) spaces (and to \cite[Part II]{BH} for a detailed discussion). Specifically, we will need the following two facts. \vskip1mm
\begin{itemize}

\item (The Bruhat-Tits Fixed Point Theorem, see \cite[Theorem 11.23]{AbBr} or \cite[\S3.2]{BT72}) {\it If $\mathcal{G}$ is a finite group acting by isometries on a CAT(0) space $X$, then the set of fixed points $X^{\mathcal{G}}$ is nonempty.}

\vskip1mm

\item (\cite[Theorem 11.16]{AbBr} and \cite[Part II, Ch. 1, Example 1.15]{BH}) {\it Finite products of CAT(0) spaces are again CAT(0) spaces. Consequently, finite products of Bruhat-Tits buildings (in particular, of Bruhat-Tits trees) are CAT(0) spaces.}

\end{itemize}

\vskip1mm

Now suppose that $\mathcal{G}$ is a group that acts on another group $A$ via
$$
\mathcal{G} \times A \longrightarrow A, \ \ (g , a) \mapsto g(a).
$$
In order to obtain information about the cohomology set $H^1(\mathcal{G} , A)$, let us assume that both $\mathcal{G}$ and $A$ admit \underline{isometric} actions on a CAT(0) space $X$:
$$
\mathcal{G} \times X \longrightarrow X, \ \ (g , x) \mapsto g(x) \ \ \ \text{and} \ \ \ A \times X \longrightarrow X, \ \ (a , x) \mapsto a(x),
$$
which are compatible in the sense that we have
\begin{equation}\label{E:compat}
g(a(x)) = (g(a))(g(x)) \ \ \text{for all} \ \ g \in \mathcal{G}, \ a \in A, \ \text{and} \ x \in X.
\end{equation}
We note that in this case, there is an isometric action of the semi-direct product $A \rtimes \mathcal{G}$ on $X$ given by $(a , g)(x) = a(g(x))$. With this set-up, we have the following.
\begin{lemma}\label{L:cohom}
Assume that $\mathcal{G}$ is finite and there exists a subset $F \subset X^{\mathcal{G}}$ such that $A \cdot F = X$ (in other words, there exists a weak fundamental domain $F$ for the action of $A$ on $X$ that is contained in the fixed point set $X^{\mathcal{G}}$).  Then every 1-cocycle $\zeta \colon \mathcal{G} \to A$ is equivalent to 1-cocycle $\zeta' \colon \mathcal{G} \to A_x$ with values in the stabilizer $A_x = \mathrm{Stab}_A(x)$ of some $x \in F$. Consequently, the map on cohomology sets
$$
\coprod_{x \in F} H^1(\mathcal{G}, A_x) \to H^1(\mathcal{G}, A),
$$
induced by the inclusions $A_x \hookrightarrow A$, is surjective.\footnotemark
\end{lemma}

\footnotetext{Since $F \subset X^{\mathcal{G}}$,  it immediately follows from (\ref{E:compat}) that for $x \in F$, the stabilizer $A_x$ is $\mathcal{G}$-invariant.}

\begin{proof}
Fix a 1-cocycle $\zeta \colon \mathcal{G} \to A$. One easily checks (cf. \cite[Lemma 3.1]{ARR}) that the map
$$
f_{\zeta} \colon \mathcal{G} \to A \rtimes \mathcal{G}, \ \ g \mapsto (\zeta(g) , g)
$$
is a group homomorphism. So, the above action of $A \rtimes \mathcal{G}$ on $X$ yields a new ``twisted'' action of $\mathcal{G}$ on $X$ given by $g * x = \zeta(g)(g(x))$. Since $\mathcal{G}$ is finite, we conclude from the Bruhat-Tits Fixed Point Theorem that there exists $x_0 \in X$ fixed by this action, hence satisfying
\begin{equation}\label{E:fixed}
\zeta(g)(g(x_0)) = x_0 \ \ \text{for all} \ \ g \in \mathcal{G}.
\end{equation}
On the other hand, since $F$ is a weak fundamental domain for the action of $A$ on $X$, we can find $a \in A$ such that $x := a^{-1}(x_0)$ belongs to $F$. Substituting $x_0 = ax$ into (\ref{E:fixed}) and taking into account that $g(x) = x$ for all $g \in \mathcal{G}$, we obtain
$$
\zeta(g)(g(a)(x)) = a(x) \ \ \text{for all} \ \ g \in \mathcal{G}.
$$
So, the 1-cocycle $\zeta'$ defined by $\zeta'(g) = a^{-1} \zeta(g) g(a)$ for $g \in \mathcal{G}$, which is equivalent to $\zeta$, takes values in $A_x$, as required.
\end{proof}
We will now quickly demonstrate how the formalism described in Lemma \ref{L:cohom} enables one to recover some conclusions made in \S\ref{S:BT}, which will pave the way for proving Theorem \ref{T:Main-quadratic}.
\begin{example}\label{Ex:local}
Let us return to the notations used in \S\ref{S:BT}. In particular, $\mathcal{K}$ will denote a local field with a fixed uniformizer $\pi$, $\mathcal{L}$ will be an unramified quadratic extension of $\mathcal{K}$ with Galois group $\mathcal{G} = \mathrm{Gal}(\mathcal{L}/\mathcal{K})$ and valuation ring $\mathcal{O}_{\mathcal{L}}$, and $G = {\rm PGL}_2.$ Let $\mathcal{T}_{\mathcal{L}}$ be the Bruhat-Tits tree constructed using $\mathcal{O}_{\mathcal{L}}$-lattices in the 2-dimensional $\mathcal{L}$-vector space $\mathcal{W}_{\mathcal{L}} = \mathcal{L}^2$ with the standard basis $e_1 , e_2$ and the standard action of $\mathcal{G}$, which induces a $\mathcal{G}$-action on $\mathcal{T}_{\mathcal{L}}$. We consider the lattices $\Lambda_i = \mathcal{O}_{\mathcal{L}} e_1 +  \mathcal{O}_{\mathcal{L}} \pi^i e_2$ for $i = 0, 1$, and let $\mathcal{P}_i = [ \Lambda_i ]$ be the corresponding vertices of $\mathcal{T}_{\mathcal{L}}$. Then the edge $\mathcal{E} = \mathcal{P}_0\mathcal{P}_1$ is a weak fundamental domain for the action of $G(\mathcal{L})$ on $\mathcal{T}_{\mathcal{L}}$; we note that $\mathcal{E}$ is contained in $\mathcal{T}_{\mathcal{L}}^{\mathcal{G}}$. Thus, Lemma \ref{L:cohom} applies, and we conclude that the map
$$
\coprod_{\mathcal{P} \in \mathcal{E}} H^1(\mathcal{G}, \Sigma_{\mathcal{L}}(\mathcal{P})) \longrightarrow H^1(\mathcal{G} , G(\mathcal{L})),
$$
where $\Sigma_{\mathcal{L}}(\mathcal{P}) = \mathrm{Stab}_{G(\mathcal{L})}(\mathcal{P})$, is surjective. We have seen (cf. Proposition \ref{P:Stab1}) that the stabilizer $\Sigma_{\mathcal{L}}(\mathcal{P})$ for $\mathcal{P} \in \mathcal{E}$ is one of the following
\vskip1mm

\begin{itemize}

\item $G(\mathcal{O}_{\mathcal{L}})$ if $\mathcal{P} = \mathcal{P}_0$;

\vskip1mm

\item $g G(\mathcal{O})g^{-1}$, with $g = \mathrm{diag}(1 , \pi)$, if $\mathcal{P} = \mathcal{P}_1$;

\vskip1mm

\item $\mathcal{F} := G(\mathcal{O}_{\mathcal{L}}) \cap \bigl( g G(\mathcal{O}_{\mathcal{L}}) g^{-1}  \bigr)$ for any $\mathcal{P} \in \mathcal{E} \setminus \{ \mathcal{P}_0, \mathcal{P}_1, \mathcal{M} \}$, where $\mathcal{M}$ is the midpoint of $\mathcal{E}$;

\vskip1mm

\item $\mathcal{F} \rtimes \langle w \rangle$, where $w = \left[ \begin{array} {cc} 0 & 1 \\ \pi & 0 \end{array}   \right]$, if $\mathcal{P} = \mathcal{M}$.

\end{itemize}

\vskip1mm

\noindent Eventually, we conclude that the natural map
$$
H^1(\mathcal{G} , G(\mathcal{O}_{\mathcal{L}})) \coprod H^1(\mathcal{G} , \mathcal{F} \rtimes \langle w \rangle) \longrightarrow H^1(\mathcal{G}, G(\mathcal{L}))
$$
is surjective. We used this fact in \S\ref{S:BT} to give a complete analysis of $H^1(\mathcal{G} , G(\mathcal{L}))$.

\end{example}

\subsection{Set-up for the proof of Theorem \ref{T:Main-quadratic}} We will prove Theorem \ref{T:Main-quadratic} by applying the approach described in Example \ref{Ex:local} in the following set-up.
Let $\ell/k$ be a (separable) quadratic extension with Galois group $\mathcal{G} = \mathrm{Gal}(\ell/k)$. We consider a 2-dimensional $k$-vector space $W_0$ with a fixed basis $e_1 , e_2$, and introduce the $k((x))$- and $k((x^{-1}))$-vector spaces
$$
W = W_0 \otimes_k k((x)) \ \ \text{and} \ \ W^- = W_0 \otimes_k k((x^{-1})),
$$
as well as the $\ell((x))$- and $\ell((x^{-1}))$-vector spaces
$$
W_{\ell} = W_0 \otimes_k \ell((x)) \ \ \text{and} \ \ W_{\ell}^- = W_0 \otimes_k \ell((x^{-1})),
$$
which we equip with the natural $\mathcal{G}$-action. We set
$$
O = k[[x]], \ \ O^- = k[[x^{-1}]], \ \ O_{\ell} = \ell[[x]], \ \ \text{and} \ \ O_{\ell}^- = \ell[[x^{-1}]].
$$
Let $\mathcal{T}_{\ell}$ (resp., $\mathcal{T}_{\ell}^-$) be the Bruhat-Tits trees constructed using $O_{\ell}$-lattices (resp., $O_{\ell}^-$-lattices) in $W_{\ell}$ (resp., $W_{\ell}^-$), and equip these with the natural $\mathcal{G}$-action. For $n \in \mathbb{Z}$, define the following lattices
$$
\Lambda_{\ell , n} = O_{\ell} e_1 + O_{\ell} x^n e_2 \ \ \text{and} \ \ \Lambda_{\ell , n}^- = O_{\ell}^- e_1 + O_{\ell}^- (x^{-1})^n e_2,
$$
and let $P_{\ell , n} = [ \Lambda_{\ell , n}]$ and $P_{\ell , n}^- = [\Lambda_{\ell , n}^-]$ denote the corresponding vertices in $\mathcal{T}_{\ell}$ and $\mathcal{T}_{\ell}^-$, respectively. Let $\mathcal{E}_{\ell , n} = P_{\ell , n} P_{\ell , n+1}$ and $\mathcal{E}_{\ell , n}^- = P_{\ell , n}^- P_{\ell , n+1}^-$ be the corresponding edges. Finally, let $\mathcal{A}_{\ell}^-$ be the apartment in $\mathcal{T}_{\ell}^-$ spanned by the vertices $\ldots , P_{\ell , -1}^-, P_{\ell , 0}^-, P_{\ell , 1}^-, \ldots $. We showed in Proposition \ref{P:FD} that $\Phi_{\ell} = \mathcal{E}_{0 , \ell} \times \mathcal{A}_{\ell}^-$ is a weak fundamental domain for the diagonal action of $\Gamma_{\ell} = {\rm PGL}_2(\ell[x , x^{-1}])$ on $X_{\ell} = \mathcal{T}_{\ell} \times \mathcal{T}_{\ell}^-$. Besides, we obviously have $\Phi_{\ell} \subset X_{\ell}^{\mathcal{G}}$. Consequently,
Lemma \ref{L:cohom} implies
that the map
$$
\coprod_{\mathcal{P} \in \Phi_{\ell}} H^1(\mathcal{G} , \Sigma_{\ell}(\mathcal{P})) \longrightarrow H^1(\mathcal{G} , \Gamma_{\ell}),
$$
where $\Sigma_{\ell}(\mathcal{P}) = \mathrm{Stab}_{\Gamma_{\ell}}(\mathcal{P})$, is surjective. We now turn to the description of the stabilizers of points $\mathcal{P} \in \Phi_{\ell}$ given in Proposition \ref{P:Stab-FD}. To deal with the possibility described in part (i) of the proposition,
we note that
any subgroup $B \subset \Gamma_{\ell}$ of the form
$$
B = \left\{ [A] \ \vert A = \left( \begin{array}{cc} a & b \\ 0 & d \end{array} \right) \ \ \text{with} \ \ a , d \in \ell^{\times}, \ b \in \mathcal{V}     \right\}
$$
or
$$
B = \left\{ [A] \ \vert A = \left( \begin{array}{cc} a & 0 \\ c & d \end{array} \right) \ \ \text{with} \ \ a , d \in \ell^{\times}, \ c \in \mathcal{V}     \right\},
$$
where $\mathcal{V} \subset \ell[x , x^{-1}]$ is a finite-dimensional $\ell$-vector space having a basis consisting of powers of $x$, is $\mathcal{G}$-invariant, and it follows from Hilbert's Theorem 90 that
$H^1(\mathcal{G} , B) = 1$. To handle the possibilities (ii) and (iii) in Proposition \ref{P:Stab-FD}, we observe that since $t = \mathrm{diag}(1 , x)$ lies in $\mathrm{GL}_2(k[x , x^{-1}])$, hence is $\mathcal{G}$-fixed, the sets
$H^1(\mathcal{G} , G(\ell))$ and $H^1(\mathcal{G} , [t] \cdot G(\ell) \cdot [t]^{-1})$ have the same image in $H^1(\mathcal{G} , \Gamma_{\ell})$. Now let $T$ be the maximal torus in $G$, which is the image of the maximal diagonal torus in $\mathrm{GL}_2$, and set $g = \left[ \begin{array}{cc} 0 & 1 \\ x & 0 \end{array} \right]$. Then, taking into account part (iv) of the proposition, we deduce that the map
\begin{equation}\label{E:Surj001}
H^1(\mathcal{G} , G(\ell)) \coprod H^1(\mathcal{G} , T(\ell) \rtimes \langle g \rangle) \longrightarrow H^1(\mathcal{G} , \Gamma_{\ell})
\end{equation}
is surjective.

\subsection{Conclusion of the proof of Theorem \ref{T:Main-quadratic}} We let $\mathcal{H}_0^*$ (resp., $\mathcal{H}_1^*$) denote the image of the map $H^1(\mathcal{G} , G(\ell)) \to H^1(\mathcal{G}, \Gamma_{\ell})$ (resp., $H^1(\mathcal{G}, T(\ell) \rtimes \langle g \rangle) \to H^1(\mathcal{G} , \Gamma_{\ell})$. Since the subsets $\mathcal{H}_0$ and $\mathcal{H}_1 \setminus \{ 1 \}$ of $H^1(\mathcal{G} , G(\ell((x))))$, introduced in subsection 2.3, are disjoint, in view of the surjectivity of  (\ref{E:Surj001}), it is enough to show that $\nu_{\ell}$ induces the bijections $\mathcal{H}_0^* \to \mathcal{H}_0$ and $\mathcal{H}_1^* \to \mathcal{H}_1$.

We have the following commutative diagram
\begin{equation}\label{E:CD001}
\begin{CD}
H^1(\mathcal{G} , G(\ell)) @ > \beta >> H^1(\mathcal{G} , G(\ell[[x]]))  \\
@ V \alpha VV @ VV \gamma V \\
H^1(\mathcal{G} , \Gamma_{\ell}) @ > \nu_{\ell} >> H^1(\mathcal{G}, G(\ell((x)))).
\end{CD}
\end{equation}
It immediately follows from Hensel's Lemma (cf. \cite[\S 3.4, Lemme 2(ii)]{BT87}) that $\beta$ is a bijection. (We note that in the case at hand, one can also prove this by
a direct argument: namely, one considers
the homomorphism $\rho \colon G(\ell[[x]]) \to G(\ell)$ given by the specialization $x \to 0$ and shows that
for the congruence subgroup $\mathcal{C} = \ker \rho$ modulo $x$, the cohomology set $H^1(\mathcal{G} , {}_{\zeta}\mathcal{C})$ is trivial for the twist of $\mathcal{C}$ by any cocycle $\zeta \in Z^1(\mathcal{G}, G(\ell[[x]]))$, which is accomplished by looking at
the filtration of $\mathcal{C}$ by the congruence subgroups modulo powers of $x$ and observing that all consecutive quotients in this filtration are
3-dimensional vector spaces, hence have trivial cohomology by the additive form of Hilbert's Theorem 90.) Furthermore, by a result of Nisnevich \cite{Nisn} (see also \cite{Guo}), the map $\gamma$ is injective. (Again, in our situation, this follows from the following well-known statement (cf. \cite[Theorem 9.6]{Saltman}): if $\mathcal{D}_1$ and $\mathcal{D}_2$ are quaternion algebras over $\mathcal{K} = k((x))$, and $\mathcal{A}_1$ and $\mathcal{A}_2$ are Azumaya algebras over $\mathcal{O} = k[[x]]$ such that $\mathcal{D}_i = \mathcal{A}_i \otimes_{\mathcal{O}} \mathcal{K}$ for $i = 1, 2$, then $\mathcal{D}_1 \simeq \mathcal{D}_2$ if and only if $\mathcal{A}_1 \simeq \mathcal{A}_2$.) These two facts immediately imply that $\nu_{\ell}$ gives a bijection between $\mathcal{H}_0^* = \im \alpha$ and $\mathcal{H}_0 = \im \gamma$, as required.

In dealing with $\mathcal{H}_1^*$ and $\mathcal{H}_1$, we will use the following subgroup introduced in subsection 2.3:
$$
\Sigma = \left\{ [A] : A = \left( \begin{array}{cc} a & b \\ c & d \end{array} \right) \in \mathrm{GL}_2(\ell[[x]]) \ \ \text{with} \ \ c \equiv 0(\mathrm{mod}\: x\ell[[x]]) \right\},
$$
which is obviously normalized by $g$. We have the following commutative diagram similar to (\ref{E:CD001}):
\begin{equation}\label{E:CD002}
\begin{CD}
H^1(\mathcal{G} , T(\ell) \rtimes \langle g \rangle) @ > \beta >> H^1(\mathcal{G} , \Sigma \rtimes \langle g \rangle)  \\
@ V \alpha VV @ VV \gamma V \\
H^1(\mathcal{G} , \Gamma_{\ell}) @ > \nu_{\ell} >> H^1(\mathcal{G}, G(\ell((x)))).
\end{CD}
\end{equation}
As above, to prove that $\nu_{\ell}$ induces a bijection between $\mathcal{H}_1^* = \im \alpha$ and $\mathcal{H}_1 = \im \gamma$, it is enough to show that in (\ref{E:CD002}), the map $\beta$ is a bijection and the map $\gamma$ is an injection.

Both facts easily follow from our considerations in subsection 2.3. Indeed, let $\ell[[x]]^{\times} \to \ell^{\times}$, $a \mapsto \bar{a}$, be the specialization $x \to 0$. Then the map $\rho \colon \Sigma \to T(\ell)$, $\left[ \begin{array}{cc} a & b \\ c & d \end{array} \right] \to \left[ \begin{array}{cc} \bar{a} & 0 \\ 0 & \bar{d} \end{array} \right]$ is the left inverse for the natural embedding $T(\ell) \to \Sigma$, with kernel
$$
\Sigma_0 =  \left\{ [A] : A = \left( \begin{array}{cc} a & b \\ c & d \end{array} \right) \in \mathrm{GL}_2(\ell[[x]]) \ \ \text{with} \ \ a,d \equiv 1(\mathrm{mod} x\ell[[x]]), \ c \equiv 0(\mathrm{mod}\: x\ell[[x]]) \right\}.
$$
We have seen in subsection 2.3 that $H^1(\mathcal{G} , {}_{\zeta}\Sigma_0)= 1$ for any cocycle $\zeta \in Z^1(\mathcal{G} , \Sigma \rtimes \langle g \rangle)$, and the bijectivity of $\beta$ follows. Furthermore, we have seen that there is a natural bijection $\omega \colon k/N_{\ell/k}(\ell^{\times}) \to H^1(\mathcal{G} , \Sigma \rtimes \langle g \rangle) \setminus \{ 1 \}$ such
that
$$
\gamma(\omega(aN_{\ell}(\ell^{\times}))) = [(\ell((x))/k((x)) , ax)],
$$
the isomorphism class of the corresponding quaternion algebra. Since the embedding $k \hookrightarrow k((x))$ yields an isomorphism
$$
k^{\times} / N_{\ell/k}(\ell^{\times}) \to k((x))^{\times}/N_{\ell((x))/k((x))}(\ell((x))^{\times}),
$$
we conclude that the element corresponding to each coset $aN_{\ell/k}(\ell^{\times}) \in k^{\times}/N_{\ell/k}(\ell^{\times})$ is nontrivial, and different cosets are mapped to different elements. This yields the injectivity of $\gamma$ and completes the proof of Theorem \ref{T:Main-quadratic}. \hfill $\Box$

\section{Proof of Theorem \ref{T:Main-sep} }\label{S:proof1.1}

In this section, we will use
the bijectivity of the map $\nu_{\ell}$ for {\it every quadratic extension} $\ell/k$, provided by Theorem \ref{T:Main-quadratic},
to establish the bijectivity of
$$
\nu_{k^{\rm sep}} \colon H^1(k^{\rm sep}/k, G(k^{\rm sep}[x, x^{-1}])) \to H^1(k((x))^{\mathrm{ur}}/ k((x)), G(k((x))^{\mathrm{ur}})).
$$
For proving both
the surjectivity and the injectivity of $\nu_{k^{\rm sep}}$, we will make use of the following fact:
\vskip2mm
\noindent $\bullet$  \parbox[t]{16cm}{\it for every cohomology class $\zeta \in H^1(k((x))^{\mathrm{ur}}/k((x)), G(k((x))^{\mathrm{ur}}))$, there exists a quadratic extension
$\ell/k$ such that the field extension $\ell((x))/k((x))$ kills $\zeta$, and then $\zeta$  lies in the image of the inflation map
$$
\iota_{\ell/k} \colon H^1(\ell((x))/k((x)), G(\ell((x)))) \longrightarrow H^1(k((x))^{\mathrm{ur}}/k((x)), G(k((x))^{\mathrm{ur}})).
$$}

\vskip2mm

\noindent Indeed,  the cohomology class $\zeta$ naturally corresponds to (the isomorphism class of) a central quaternion $k((x))$-algebra $D$. It is well-known that $D$ contains a maximal unramified subfield (cf. \cite[Theorem 10.1]{Saltman}) , which in our case is necessarily of the form $\ell((x))$ for some quadratic extension $\ell/k$. Then the base change $\ell((x))/k((x))$ kills the Brauer class of $D$, and hence the cohomology class $\zeta$ itself. Then the fact that $\zeta$ lies in $\im \iota_{\ell/k}$ follows from the (noncommutative) Inflation-Restriction Sequence (see \cite[Ch. 1, \S1.3.2]{PlRR}).

\vskip2mm

\noindent $\bullet$ \underline{$\nu_{k^{\rm sep}}$ {\it is surjective}.} For a given $\zeta \in H^1(k^{\rm sep}((x))/k((x)), G(k^{\rm sep}((x))))$, pick a quadratic extension $\ell/k$ so that $\zeta \in \im \iota_{\ell/k}$. Since $\nu_{\ell}$ is surjective, it follows
that $\zeta$ lies in the image of the composite map
$$
H^1(\ell/k, G(\ell[x, x^{-1}])) \stackrel{\nu_{\ell}}{\longrightarrow} H^1(\ell((x))/k((x)), G(\ell((x)))) \stackrel{\iota_{\ell/k}}{\longrightarrow} H^1(k((x))^{\mathrm{ur}}/k((x)), G(k^((x))^{\mathrm{ur}})),
$$
hence in the image of $\nu_{k^{\rm sep}}$.

\vskip2mm

\noindent $\bullet$ \underline{$\nu_{k^{\rm sep}}$ {\it is injective}.} We begin with the following.
\begin{lemma}\label{L:TriK}
For any finite Galois extension $k'/k$, the natural map
$$
\nu_{k'} \colon H^1(k'/k , G(k'[x , x^{-1}])) \longrightarrow H^1(k'((x))/k((x)) , G(k'((x)))
$$
has trivial kernel.
\end{lemma}
\begin{proof}
Set $\Gamma_{k'} = G(k'[x , x^{-1}])$. First, let us show that
\begin{equation}\label{E:Decom1}
G(k'((x))) = \Gamma_{k'} \cdot G(k'[[x]]).
\end{equation}
Indeed, since $k'[x , x^{-1}]$ is dense in $k'((x))$, considering elementary matrices, we find that the closure of $\mathrm{GL}_2(k'[x , x^{-1}])$ in $\mathrm{GL}_2(k'((x)))$ contains $\mathrm{SL}_2(k'((x)))$, and therefore
\begin{equation}\label{E:Decom2}
\mathrm{SL}_2(k'((x))) \subset \mathrm{GL}_2(k'[x , x^{-1}]) \cdot \mathrm{GL}_2(k[[x]]).
\end{equation}
On the other hand, let $T$ be the diagonal torus in $\mathrm{GL}_2$. Then the decomposition $k'((x))^{\times} = \langle x \rangle \cdot k'[[x]]^{\times}$
yields the decomposition $T(k'((x))) = T(k'[x , x^{-1}]) \cdot T(k'[[x]])$. Combining this with (\ref{E:Decom2}), we obtain
$$
\mathrm{GL}_2(k'((x))) = \mathrm{SL}_2(k'((x))) \cdot T(k'((x))) = \mathrm{SL}_2(k'((x))) \cdot T(k'[x , x^{-1}]) \cdot T(k'[[x]]) =
$$
$$
=T(k'[x , x^{-1}]) \cdot \mathrm{SL}_2(k'((x))) \cdot T(k'[[x]]) = \mathrm{GL}_2(k'[x , x^{-1}]) \cdot \mathrm{GL}_2(k'[[x]]),
$$
and (\ref{E:Decom1}) follows.

Now suppose $\zeta \in Z^1(k'/k , \Gamma_{k'})$ is a cocycle such that the corresponding cohomology class lies in $\ker \nu_{k'}$. This means that there exists $g \in G(k'((x)))$ such that
$$
\zeta(\sigma) = g^{-1} \sigma(g) \ \ \text{for all} \ \ \sigma \in \mathrm{Gal}(k'/k).
$$
According to (\ref{E:Decom1}), we can write $g^{-1} = h \cdot s$ with $h \in \Gamma_{k'}$ and $s \in G(k'[[x]])$. Then the equation
\begin{equation}\label{E:Cocyc1}
\zeta'(\sigma) := h^{-1} \zeta(\sigma) \sigma(h) = s \sigma(s)^{-1} \ \ \text{for} \ \ \sigma \in \mathrm{Gal}(k'/k)
\end{equation}
defines a 1-cocycle  $\zeta' \in Z^1(k'/k , \Gamma_{k'})$ which is equivalent to $\zeta$ but has values in $\Gamma_{k'} \cap G(k'[[x]]) = G(k'[x])$. Since $\zeta'$ becomes trivial over the separable extension $k'/k$,  it follows from the Raghunathan-Ramanathan theorem \cite{RR} that $\zeta'$ is equivalent in $Z^1(k'/k , G(k'[x]))$ to a cocycle $\zeta_0 \in Z^1(k'/k , G(k'))$. According to (\ref{E:Cocyc1}), the cocycle $\zeta_0$ becomes trivial in $H^1(k'/k , G(k'[[x]]))$, and using the specialization $k'[[x]] \to k'$, $x \to 0$, we conclude that $\zeta_0$ is trivial in $H^1(k'/k , G)$, completing the argument.
\end{proof}

\begin{remark}
With minor modifications, the proof of Lemma  \ref{L:TriK} extends to an arbitrary $k$-split reductive group (cf. the proof of Theorem ..).
\end{remark}

We are now in a position to finish the proof of the injectivity of $\nu_{k^{\mathrm sep}}$. Suppose we have cohomology classes
$\xi_1 , \xi_2 \in H^1(k^{\mathrm{sep}}/k , G(k^{\mathrm{sep}}[x , x^{-1}]))$
such that
\begin{equation}\label{E:Equal1}
\nu_{k^{\mathrm{sep}}}(\xi_1) = \nu_{k^{\mathrm{sep}}}(\xi_2) =: \zeta.
\end{equation}
As noted above,
we can pick a quadratic extension $\ell/k$ such that the base change $\ell((x))/k((x))$ kills $\zeta$.
We have the following commutative diagram
$$
\begin{CD}
H^1(k^{\mathrm{sep}}/k , G(k^{\mathrm{sep}}[x , x^{-1}])) @ > \nu_{k^{\mathrm{sep}}} >> H^1(k((x))^{\mathrm{ur}}/k((x)) , G(k((x))^{\mathrm{ur}})) \\
@ V \rho_1 VV   @ VV \rho_2 V \\
H^1(\ell^{\mathrm{sep}}/\ell , G(\ell^{\mathrm{sep}}[x , x^{-1}])) @ > \nu_{\ell^{\mathrm{sep}}} >> H^1(\ell((x))^{\mathrm{ur}} / \ell((x)) , G(\ell((x))^{\mathrm{ur}}))
\end{CD}
$$
where $\rho_1$ and $\rho_2$ are the corresponding restriction maps (we note that of course $\ell((x))^{\mathrm{ur}} = k((x))^{\mathrm{ur}}$). By construction, $\rho_2(\zeta)$ is trivial. On the other hand, according to
Lemma \ref{L:TriK}, applied over $\ell$, the map
$\nu_{\ell^{\mathrm{sep}}}$ has trivial kernel, implying that both $\rho_1(\xi_1)$ and $\rho_1(\xi_2)$ are trivial. Then it follows from the Inflation-Restriction Sequence that we can write $\xi_i = \varepsilon_{\ell/k}(\theta_i)$ for $i = 1, 2$, where
$$
\varepsilon_{\ell/k} \colon H^1(\ell/k , G(\ell[x , x^{-1}])) \longrightarrow H^1(k^{\mathrm{sep}}/k , G(k^{\mathrm{sep}}[x , x^{-1}]))
$$
is the inflation map, and $\theta_i \in H^1(\ell/k , G(\ell[x , x^{-1}]))$.

Finally, in the commutative diagram
$$
\begin{CD}
H^1(\ell/k , G(\ell[x , x^{-1}])) @ > \nu_{\ell} >> H^1(\ell((x))/k((x)) , G(\ell((x)))) \\
@ V \varepsilon_{\ell/k} VV  @ VV \iota_{\ell/k} V \\
H^1(k^{\mathrm{sep}}/k , G(k^{\mathrm{sep}}[x , x^{-1}])) @ > \nu_{k^{\mathrm{sep}}} >> H^1(k((x))^{\mathrm{ur}}/ k((x)) , G(k((x))^{\mathrm{ur}}))
\end{CD}
$$
the maps $\nu_{\ell}$ and $\iota_{\ell/k}$ are injective, so (\ref{E:Equal1}) implies that $\theta_1 = \theta_2$, hence $\xi_1 = \xi_2$, as required. \hfill $\Box$

\section{Finite subgroups of $\Gamma$}\label{S:FS}

Using actions on appropriate affine buildings, it was proved in \cite[Theorem 1.1]{ARR} that for an arbitrary reductive algebraic group $G$ over a field $k$ of characteristic zero, every finite subgroup of $G(k[x])$ is conjugate to a subgroup contained in $G(k)$. It is easy to see (cf. Remark \ref{R:FS1} below) that this statement is no longer valid over the ring $k[x,x^{-1}]$ of Laurent polynomials, but nevertheless it has the following analogue over this ring for the group $G = \mathrm{PGL}_2$.

Suppose $k$ is again a field of characteristic 0, and let
$G = \mathrm{PGL}_2$, $\Gamma = G(k[x , x^{-1}])$, and $\Sigma = T(k) \rtimes \langle [g] \rangle$, where $T$ is the image in $G$ of the diagonal torus in $\mathrm{GL}_2$ and $g = \left( \begin{array}{cc} 0 & 1 \\ x & 0 \end{array}  \right)$.
\begin{thm}\label{T:FS1}
 Every finite subgroup of $\Gamma$ is conjugate in $\Gamma$ to either a subgroup of $G(k)$ or a subgroup of $\Sigma$.
\end{thm}
\begin{proof}
The argument closely follows the proof of \cite[Theorem 1.1]{ARR}. Let $\mathcal{T}$ and $\mathcal{T}^-$ be the Bruhat-Tits trees associated with $k((x))$ and $k((x^{-1}))$, and $\mathcal{X} = \mathcal{T} \times \mathcal{T}^-$ as in \S \ref{S:Product}. Then the group $\Gamma$ acts on each tree by isometries, and the diagonal action of $\Gamma$ on $\mathcal{X}$ admits a weak fundamental domain $\Phi$ described in subsection 3.1 (see Proposition \ref{P:FD}).   Let $\Delta \subset \Gamma$ be a finite subgroup.   Applying the Bruhat-Tits Fixed Point Theorem to the action of $\Delta$ on the CAT(0) space $\mathcal{X}$, we obtain that $\Delta$ fixes a point $\mathcal{P} \in \mathcal{X}$. Replacing $\Delta$ with a $\Gamma$-conjugate subgroup, we may assume that $\mathcal{P} \in \Phi$. Thus, $\Delta  \subset \mathrm{Stab}_{\Gamma}(\mathcal{P})$, and all the possibilities for the stabilizer $\mathrm{Stab}_{\Gamma}(\mathcal{P})$, with $\mathcal{P} \in \Phi$, are listed in the items (i)-(iv) of Proposition \ref{P:Stab-FD}. The stabilizer in case (i) is $G(k)$, in case (ii), it is conjugate to $G(k)$, and in case (iv), it is $\Sigma$. So, in these cases, the required assertion follows directly. It remains to consider case (i). Here we have
$$
\mathrm{Stab}_{\Gamma}(\mathcal{P}) = U(k) \rtimes T(k),
$$
where $T$ is the image in $G$ of the maximal diagonal torus in $\mathrm{GL}_2$, and $U$ is a vector $k$-group (i.e. a direct product of copies of the additive group $\mathbb{G}_a$) with a $k$-defined action of $T$.  Since $\mathrm{char}\: k = 0$, according to \cite[Corollary 2.4]{ARR}, every finite subgroup of $U(k) \rtimes T(k)$ is conjugate to a subgroup of $T(k) \subset G(k)$, completing the argument.
\end{proof}

\begin{remark}\label{R:FS1}
The order two subgroup $\langle [g] \rangle \subset \Gamma$ is {\it not} conjugate to a subgroup of $G(k)$. Indeed, we have a group homomorphism
$$
\psi \colon G(k(x)) \to k(x)^{\times} / {k(x)^{\times}}^2,  \ \ \psi([A]) = (\det A) {k(x)^{\times}}^2.
$$
Our claim follows from the observations that
$$
\psi(G(k)) = (k^{\times} {k(x)^{\times}}^2)/{k(x)^{\times}}^2 \ \ \text{and} \ \ \psi(\langle [g] \rangle) = \langle x {k(x)^{\times}}^2 \rangle.
$$
\end{remark}

As in \cite{ARR}, we say that a group possesses property (FC) if it has finitely many conjugacy classes of finite subgroups.
\begin{cor}\label{C:FC1}
 Let $G = \mathrm{PGL}_2$ over a finite extension $k$ of the $p$-adic field $\mathbb{Q}_p$. Then the group $\Gamma = G(k[x , x^{-1}])$ has property (FC).
\end{cor}
\begin{proof}
It was shown in \cite[Proposition 4.3]{ARR} that the group $G(k)$ does have (FC). So, in view of
Theorem \ref{T:FS1}, it remains to prove that $\Sigma = T(k) \rtimes \langle [g] \rangle$ also has (FC). Since $k$ contains only finitely many roots of unity, the set of finite subgroups of the group $T(k)$ is finite. So, it is enough to show that if we fix one of them, say, $H$, then the set of subgroups $H \subsetneqq S \subset \Sigma$ such that $S \cap T(k) = H$ consists of finitely many conjugacy classes under $T(k)$. Since $[S : H] = 2$, this would follow if we show that $\Sigma \setminus T(k) = T(k) \cdot [g]$ consists of finitely many $T(k)$-conjugacy classes. But for $s , t \in T(k)$, we have
$$
t (s [g] ) t^{-1} = t^2s [g],
$$
so these classes are in bijection with the elements of $T(k)/T(k)^2 = k^{\times}/{k^{\times}}^2$, which is finite.
\end{proof}

To conclude this section, we will discuss similar (although not quite identical) results for $\Gamma_0 = \mathrm{GL}_2(k[x , x^{-1}])$ and $\Gamma_1 = \mathrm{SL}_2(k[x , x^{-1}])$.

\vskip2mm

 $\bullet$ $G_0 = \mathrm{GL}_2$. Here, every finite subgroup of $\Gamma_0 = G_0(k[x , x^{-1}])$ is conjugate to a subgroup of $G_0(k)$. This fact can be derived from Theorem \ref{T:FS1} by the following elementary group-theoretic argument. Let $\pi \colon \Gamma_0 \to \Gamma$ be the canonical homomorphism. By Theorem \ref{T:FS1}, replacing a given finite subgroup $\Delta_0 \subset \Gamma_0$ with a conjugate one, we may assume that $\pi(\Delta_0)$ is contained in either $G(k)$ or $\Sigma$. Then $\Delta_0$  is contained in either $\pi^{-1}(G(k)) = G_0(k) \cdot \langle xI_2 \rangle$ or
 $$
 \pi^{-1}(\Sigma) = \left\{ \left( \begin{array}{cc} a & 0 \\ 0 & b    \end{array} \right) :  a,b \in k^{\times} \cdot \langle x \rangle   \right\} \cdot \left\langle \left(  \begin{array}{cc} 0 & 1 \\ x & 0 \end{array}  \right) \right\rangle =  \left\{ \left( \begin{array}{cc} a & 0 \\ 0 & b    \end{array} \right) :  a,b \in k^{\times} \cdot \langle x \rangle   \right\} \rtimes \langle w \rangle,
 $$
where $w = \left( \begin{array}{cc} 0 & 1 \\ 1 & 0 \end{array} \right)$.  If $\Delta_0 \subset T_0(k(x))$, where $T_0 \subset G_0$ is the diagonal torus, then by order considerations $\Delta_0 \subset T_0(k)$. Otherwise, $H := \Delta_0 \cap T_0(k(x))$ is an index 2 subgroup of $\Delta_0$.
Take any $h \in \Delta_0 \setminus H$. Then $h = t w$, with $t = \mathrm{diag}(a , b)$ and $a , b \in k^{\times} \cdot \langle x \rangle$, and $h^{2n} = I_2$ for some $n \geq 1$.
Since
$h^{2n} = (abI_2)^n$, it follows that $ab \in k^{\times}$. If $a = \alpha x^{\ell}$ ($\alpha \in k^{\times}$), then for $r = \mathrm{diag}(x^{\ell} , 1)$, we have $r^{-1} h r \in G_0(k)$, and then $r^{-1} \Delta_0 r \subset G_0(k)$, as required.

\vskip2mm

$\bullet$ $G_1 = \mathrm{SL}_2$. We will show that every finite subgroup  $\Delta_1 \subset \Gamma_1$ is conjugate to a subgroup of either $G_1(k)$ or $t G_1(k) t^{-1}$, where $t = \mathrm{diag}(1 , x)$. While this statement can be proven by a purely algebraic argument, it is more illuminating in this case to use the action on $\mathcal{X} = \mathcal{T} \times \mathcal{T}^-$. The point is that $\Gamma_1$ acts on both of the trees $\mathcal{T}$ and $\mathcal{T}^-$ without inversions (cf. \cite[Ch. II, \S1.3]{Serre-Trees}). So, we conclude from the Bruhat-Tits Fixed Point Theorem that a given finite subgroup $\Delta_1$ fixes a point $\mathcal{P} = (P , P^-) \in \mathcal{X}$, where both $P$ and $P^-$ are vertices (cf. \cite[Ch. II, \S1.3, Proposition 2]{Serre-Trees}). Furthermore, replacing $\Delta_1$ with a conjugate subgroup, we may assume that $\mathcal{P} \in \Phi$. Then the computations made in Cases 1 and 2 of the proof of Proposition \ref{P:Stab-FD} demonstrate that $\mathrm{Stab}_{\Gamma_1}(\mathcal{P})$ is one of the following: $G_1(k)$, $t G_1(k) t^{-1}$, or $U_1(k) \rtimes T_1(k)$, where $T_1$ is the diagonal torus in $G_1$ and $U_1$ is a vector $k$-group with a $T_1$-action. Again, using \cite[Corollary 2.4]{ARR}, we conclude that if $\Delta_1 \subset U_1(k) \rtimes T_1(k)$, then $\Delta_1$ is conjugate to a subgroup of $T_1(k) \subset G_1(k)$  completing the argument.

\vskip2mm

A direct application of \cite[Proposition 3.3]{ARR} now shows that both $\Gamma_0$ and $\Gamma_1$ have property (FC).

\section{On the kernel of the global-to-local map}\label{S:LG}

We will now
discuss applications of the Raghunathan-Ramanathan Theorem and of our Theorem \ref{T:Main-sep} to the local-global principle. We recall that given a field $K$ equipped with a set $V$ of (rank one) valuations of $K$ and a linear algebraic $K$-group $G$, one defines the corresponding Tate-Shafarevich set  $\Sha(K, V, G)$ as the kernel of the natural map
$$
\lambda_{K, V, G} \colon H^1(K , G) \longrightarrow \prod_{v \in V} H^1(K_v , G).
$$

In this section, we will be mainly concerned with the case of the rational function field $K = k(x)$, where, for the simplicity of the statements, we will assume that the field of constants $k$ is perfect.
Let $V_0$ be the set of the discrete valuations $v = v_{p(x)}$ of $K$ associated with the monic irreducible polynomials $p(x) \in k[x]$ (in geometric terms, this is the set of valuations corresponding to the closed points of the affine line $\mathbb{A}^1_k \subset \mathbb{P}^1_k$).
We then have the following result.
\begin{thm}\label{T:LG1}
For any connected reductive $k$-group $G$, the Tate-Shafarevich set $\Sha(K, V_0, G)$ is trivial.   \end{thm}
\begin{proof}
Our argument will rely
on adelic considerations, which have been used systematically in the analysis of local-global phenomena over arbitrary fields since \cite{CRR-Isr}, \cite{RR-NT}, and \cite{RR-MRL},
and we refer the reader to these papers for a detailed description of the set-up.
In particular, for a field $K$ equipped with a set of valuations $V$, we will denote by
$\mathbb{A}(K , V)$ the corresponding
adele ring of $K$, i.e.
the restricted product of the completions $K_v$ for $v \in V$ with respect to the valuation rings $\mathcal{O}_v \subset K_v$ (we will only need the case where all valuations in $V$ are nonarchimedean).  Assuming that $V$ satisfies the following condition
\vskip2mm

\noindent (A) \ for any  $a \in K^{\times}$, the set $V(a) := \{ v \in V \, \vert \, v(a) \neq 0 \}$ is finite,

\vskip2mm
\noindent we will identify $K$ with the subring of principal adeles in $\mathbb{A}(K , V)$, and let
$$
\mathbb{A}^{\infty}(K , V) = \prod_{v \in V} \mathcal{O}_v
$$
denote the subring of integral adeles. Now let $L/K$ be a finite Galois extension, and let $V^L$ denote the set of extensions of all valuations $v \in V$ to $L$. Then the canonical embeddings $\mathbb{A}(K, V) \to \mathbb{A}(L , V^L)$ and $L \to \mathbb{A}(L, V^L)$ give rise to an isomorphism
$$
\mathbb{A}(K, V) \otimes_K L \simeq \mathbb{A}(L , V^L),
$$
making $\mathbb{A}(L, V^L)$ into a $\mathrm{Gal}(L/K)$-module.

Henceforth, we let $K = k(t)$. If $\bar{k}$ is an algebraic closure of $k$, then the compositum $\bar{k}K = \bar{k}(x)$ has cohomological dimension $\leq 1$ by Tsen's Theorem (see \cite[Proposition 6.2.3 and Theorem 6.2.8]{GS}), so using Steinberg's Theorem (cf. \cite[Ch. III, \S2.3]{Serre-GC} ), we conclude that $H^1(\bar{k}K , G) = 1$. As $k$ is perfect, the extension $\bar{k}K/K$ is Galois, and the Inflation-Restriction Sequence yields the equality $H^1(K, G) = H^1(\bar{k}K/K , G)$. Thus, it is enough to show that given a finite Galois extension $\ell/k$ (which can be assumed to split $G$), we have $H^1(\ell K/K , G) \cap \ker \lambda_{K, V_0, G} = 1$.

Fix a finite Galois extension $\ell/k$ that splits $G$
and set $L = \ell K$. We note the following.
\begin{lemma}\label{L:LG1}
$G(\mathbb{A}(L , V_0^L)) = G(\mathbb{A}^{\infty}(L , V_0^L)) G(L)$.
\end{lemma}
\begin{proof}
(cf. \cite[\S 4]{CRR-Isr}) Let $T$ be a maximal $\ell$-split torus in $G$. It was shown in \cite{CRR-Isr} using strong approximation that
$$
G(\mathbb{A}(L, V_0^L)) = G(\mathbb{A}^{\infty}(L , V_0^L)) T(\mathbb{A}(L, V_0^L)) G(L).
$$
On the other hand, $T \simeq (\mathbb{G}_m)^{\dim G}$ (over $\ell$), and for the multiplicative group $\mathbb{G}_m$, we have
$$
\mathbb{A}(L , V_0^L)^{\times} = \mathbb{A}^{\infty}(L , V_0^L)^{\times} \cdot L^{\times}
$$
as $\ell[x]$ is a unique factorization domain (cf. \cite[Lemma 2.2]{CRR-Spinor}). So, $T(\mathbb{A}(L , V_0^L)) = T(\mathbb{A}^{\infty}(L, V_0^L)) T(L)$, and the required fact follows.
\end{proof}

Next, we have

\begin{lemma}\label{L:LG2}
$H^1(L/K , G) \cap \ker \lambda_{K, V_0, G} = \ker\bigl( H^1(L/K , G) \to H^1(L/K , G(\mathbb{A}(L, V_0^L))   \bigr)$.
\end{lemma}
\begin{proof}
The inclusion $\supset$ is obvious. For the opposite inclusion, we need to observe that, being defined over $k$, the group $G$ has good/reductive reduction at all $v \in V_0$, and, moreover, the extension $L = \ell K$ is unramifed at all $v \in V_0$. So, it follows from a result of Nisnevich \cite{Nisn} (see also \cite{Guo}) that for any $v \in V_0$ and $w \vert v$, the map
$$
\iota_w \colon H^1(L_w/K_v , G(\mathcal{O}_{L_w})) \to H^1(L_w/K_v , G(L_w)),
$$
where  $\mathcal{O}_{L_w}$ is the valuation ring in $L_w$, has trivial kernel. Let $\zeta \in Z^1(L/K , G)$ be a 1-cocycle whose cohomology class defines an element of $H^1(L/K , G) \cap \ker \lambda_{K, V_0, G}$. The values of $\zeta$ lie in $G(\mathcal{O}_{L_w})$ for almost all $v \in V_0$ and $w \vert v$. It follows from the fact that  $\ker \iota_w$ is trivial that for those $v$, the image of the cohomology class of $\zeta$ in $H^1(L_w/K_v , G(\mathcal{O}_{L_w}))$ is trivial. This implies that the image of this class in $H^1(L/K , G(\mathbb{A}(L , V_0^L)))$ is trivial, as required.
\end{proof}

We will now complete
the proof of Theorem \ref{T:LG1}. Let $\zeta \in Z^1(L/K , G)$ be a cocycle such that the corresponding cohomology class lies in $H^1(L/K , G) \cap \ker \lambda_{K, V_0, G}$. It follows from Lemmas \ref{L:LG1} and \ref{L:LG2} that there exist $a \in G(\mathbb{A}^{\infty}(L , V_0^L))$ and $b \in G(L)$ such that
$$
\zeta(\sigma) = (ab)^{-1} \sigma(ab) \ \ \text{for all} \ \ \sigma \in \mathrm{Gal}(L/K).
$$
Then for any $\sigma \in \mathrm{Gal}(L/K)$, we have
$$
b\zeta(\sigma) \sigma(b)^{-1} = a^{-1} \sigma(a) \in G(L) \cap G(\mathbb{A}^{\infty}(L , V_0^L)) = G(\ell[x]).
$$
Thus, the cohomology class of $\zeta$ lies in the image of the map $H^1(\ell/k , G(\ell[x])) \to H^1(L/K , G(L))$. However, $H^1(\ell/k , G(\ell[x])) = H^1(\ell/k , G(\ell))$ by the Raghunathan-Ramanathan Theorem (cf. \cite{ARR}, \cite{RR}). So, without loss of generality, we may assume that $\zeta \in Z^1(\ell/k , G(\ell))$. Furthermore, the cohomology class of $\zeta$ lies in
$$
\ker\bigl(H^1(\ell/k , G(\ell)) \to H^1(\ell/k , G(\ell((x))))\bigr).
$$
Since the map $H^1(\ell/k , G(\ell[[x]])) \to H^1(\ell/k , G(\ell((x)))$ has trivial kernel (as noted above), we conclude that this class actually lies in
$$
\ker\bigl(H^1(\ell/k , G(\ell)) \to H^1(\ell/k , G(\ell[[x]])) \bigr).
$$
But the specialization $\ell[[x]] \to \ell$, $x \mapsto 0$, defines a map $H^1(\ell/k , G(\ell[[x]])) \to H^1(\ell/k , G(\ell))$ such that the composition
$$
H^1(\ell/k , G(\ell)) \to H^1(\ell/k , G(\ell[[x]])) \to H^1(\ell/k , G(\ell))
$$
is the identity map. It follows that the cohomology class of $\zeta$ is trivial, proving the theorem.
\end{proof}

\begin{remark}
In general, the fact that the global-to-local map in Galois cohomology has trivial kernel may not imply that it is actually injective, i.e. that the local-global principle holds. However, the triviality of the kernel does imply the injectivity in certain situations where a suitable group structure is present, e.g. when dealing with the Brauer group.
\end{remark}

It should be pointed out that the properties of the local-to-global map change if we replace $V_0$ with $V_1 := V_0 \setminus \{ v_x \}$, even for $G = \mathrm{PGL}_2$. Indeed, let $k$ be a {\it finite} field. By global class field theory, there exists a unique (up to isomorphism) central quaternion algebra $D$ over $K = k(x)$ that ramifies at $v_x$ and $v_{x^{-1}} = v_{\infty}$ (cf. \cite[\S6.5]{GS}). Then the (nontrivial) element $\zeta \in H^1(K , G)$ that corresponds to the isomorphism class of $D$ lies in $\ker \lambda_{K, V_1, G}$, making the kernel nontrivial. A similar, but even simpler, example can be constructed over the infinite field $k = \mathbb{C}((t))$: it suffices to consider the quaternion algebra $\displaystyle \left( \frac{t , x}{K}  \right)$, where $K = k(x)$. Nevertheless, we have the following.

\begin{thm}\label{T:LG2}
Let $k$ be either a $p$-adic field or a finitely generated field of characteristic zero. Then, in the above notations, for $G = \mathrm{PGL}_2$, the kernel $\ker \lambda_{K, V_1, G}$ is trivial.   \end{thm}

(The assertion remains valid over finitely generated fields of characteristic > 2 having transcendence degree > 0 over the prime subfield.)

\begin{proof}
While the overarching structure of the argument is similar to the proof of Theorem \ref{T:LG1}, we will  use along the way various specific results for the situation at hand established in \S\S 4-5. First, it is enough to show that for any finite Galois extension  $\ell/k$, the intersection
$$
X(\ell) := H^1(\ell K / K , G) \cap \ker \lambda_{K, V_1, G}
$$
is trivial. Since $G$ is $k$-split and $\ell[x , x^{-1}]$ is a unique factorization domain, the argument in the proof of Theorem \ref{T:LG1} demonstrates that $X(\ell)$ is contained in
$$
Y(\ell) :=  \im \bigl( H^1(\ell/k , \Gamma_{\ell}) \to H^1(\ell K / K , G) \bigr),
$$
where $\Gamma_{\ell} = G(\ell[x , x^{-1}])$.

Since every element of $H^1(k((x))^{\mathrm{ur}}/k((x)) , G(k((x))^{\mathrm{ur}}))$ becomes trivial
over $\ell((x))$ for some quadratic extension $\ell/k$, it follows from Theorem \ref{T:Main-sep} that every element of $H^1(\bar{k}/k , G(\bar{k}[x , x^{-1}]))$ becomes trivial over such $\ell$ --- we recall that this
was actually part of the argument in \S 5. This means that it is enough to prove that for every \emph{quadratic} extension $\ell/k$, the intersection
\begin{equation}\label{E:Intersec1}
Y(\ell) \cap \ker \lambda_{K, V_1, G}
\end{equation}
is trivial.

For the rest of the argument, let us fix a quadratic extension $\ell/k$. We have seen in subsection 4.3 that $H^1(\ell/k , \Gamma_{\ell})$ is the union of the sets $\mathcal{H}_0^*$ and $\mathcal{H}_1^*$ introduced there. We recall that $\mathcal{H}_0^*$ is the image of the map
$H^1(\ell/k , G) \to H^1(\ell/k , \Gamma_{\ell})$. Now, if $\zeta \in H^1(\ell/k , G)$ defines an element that lies in the intersection (\ref{E:Intersec1}), then in particular, $\zeta$ lies in the kernel of the map $H^1(\ell/k , G) \to H^1(\ell(x-1))/k((x-1)) , G)$, where $k((x-1))$ is understood as the completion of $K = k(x)$ with respect to the valuation $v_{(x-1)}$ associated with $x-1$. Then the same argument as in the proof of Theorem \ref{T:LG1} yields that $\zeta = 1$.

Next, by definition $\mathcal{H}_1^* = \im\bigl( H^1(\ell/k , T(\ell) \rtimes \langle g \rangle)  \to H^1(\ell/k , \Gamma_{\ell})\bigr)$, where $T$ is the image in $G$ of the diagonal torus in $\mathrm{GL}_2$ and $g = \left[  \begin{array}{cc} 0 & 1 \\ x & 0   \end{array} \right]$. We have seen in subsection 4.3 that every $\zeta \in H^1(\ell((x))/k((x)) , G)$ that comes from $\mathcal{H}_1^* \setminus \{ 1 \}$ corresponds to a quaternions algebra of the form $D = (\ell(x)/k(x) , ax)$ for some $a \in k^{\times}$. Such $D$ is unramified at $v_{(x - \alpha)}$ for any $\alpha \in k^{\times}$, with the residue given by the quaternion algebra $(\ell/k , a\alpha)$. So, if $\zeta \in \ker \lambda_{K, V_1, G}$, then the algebra $(\ell/k , a\alpha)$ must be trivial for all $\alpha \in k^{\times}$, implying that $N_{\ell/k}(\ell^{\times}) = k^{\times}$ (cf. \cite[Proposition 1.1.7]{GS}). Thus, it remains to prove the following.
\begin{lemma}\label{L:Norms}
 Let $k$ be either a $p$-adic field or a finitely generated field of characteristic zero. Then for any quadratic extension $\ell/k$, we have $N_{\ell/k}(\ell^{\times}) \neq k^{\times}$.
\end{lemma}

Indeed, if $k$ is a $p$-adic field, then $[k^{\times} : N_{\ell/k}(\ell^{\times})] = 2$ by local class field theory (cf. \cite[Ch. XIII, \S4]{Serre-LF}). Now suppose that $k$ is finitely generated. Then according to \cite[Proposition 2.1]{PrR}, for some prime $p$, there exists an embedding $k \hookrightarrow \mathcal{K} : = \mathbb{Q}_p$ such that $\mathcal{L} := \ell \mathcal{K}$ is a quadratic extension of $\mathcal{K}$. Since $k^{\times}$ is dense in $\mathcal{K}^{\times}$ in the $p$-adic topology, the equality $N_{\ell/k}(\ell^{\times}) = k^{\times}$ would imply that $N_{\mathcal{L}/\mathcal{K}}(\mathcal{L}^{\times}) = \mathcal{K}^{\times}$ as the norm subgroup $N_{\mathcal{L}/\mathcal{K}}(\mathcal{L}^{\times})$ is open, hence closed. But this contradicts the first part of the argument.
\end{proof}

\begin{remark}
By \cite[Proposition 2.1]{PrR}, for a given finitely generated field $k$, the set of primes with the properties used in the above argument is infinite. So, using weak approximation, one proves that in this case, the index $[k^{\times} : N_{\ell/k}(\ell^{\times})]$ is infinite.
\end{remark}

\vskip4mm

\noindent {\bf Acknowledgements.} Igor Rapinchuk and Avinash Roy were partially supported by NSF grant DMS-2154408 during the preparation of this paper.

\bibliographystyle{amsplain}

\end{document}